\documentclass{amsart}
\usepackage[final]{microtype}
\usepackage{graphicx}
\usepackage{amsmath,amsfonts,amssymb}
\usepackage[pdfusetitle,pdfborder={0 0 .5}]{hyperref}
\usepackage{cleveref}

\usepackage[textsize=tiny]{todonotes}
\setuptodonotes{fancyline, color=blue!30}

\theoremstyle{plain}
\newtheorem{theorem}{Theorem}[section]
\newtheorem{thm}[theorem]{Theorem}
\newtheorem{lemma}[theorem]{Lemma}

\newtheorem{corollary}[theorem]{Corollary}

\newtheorem{prop}[theorem]{Proposition}

\newtheorem*{claim*}{Claim}
\newtheorem*{theorem:main}{Main Theorem}

\theoremstyle{definition}
\newtheorem{definition}[theorem]{Definition}
\newtheorem{remark}[theorem]{Remark}
\newtheorem{example}[theorem]{Example}
\newtheorem{assumption}[theorem]{Assumption}

\newcommand{\vertx}[1]{\mathcal{V}_{#1}}

\newcommand{\length}[1]{{\mathrm{Len}({#1})}}
\newcommand{\infsub}{{\scriptscriptstyle{(\infty)}}}
\newcommand{\infcohom}[2]{H^2_\infsub(#1;#2)}
\newcommand{\cusp}[1]{\mathrm{Cusp}(#1)}
\newcommand{\bd}{\partial}
\newcommand{\cbd}{\delta}

\newcommand{\R}{\mathbb R}

\newcommand{\cay}[1]{\mathrm{Cay}(#1)}
\newcommand{\chain}[1]{\langle{#1}\rangle}
\newcommand{\rarea}[2]{\mathrm{Area}_{#1}({#2})}
\newcommand{\linfn}{{\ell_\infty}}
\newcommand{\norm}[1]{{\lvert{} {#1} \rvert{}}}
\newcommand{\prel}[2]{{\mathrm{Rel}({#1},{#2})}}

\title[Cohomological characterisation of hyperbolicity]{Cohomological characterisation of hyperbolicity}

\author[Milizia]{Francesco Milizia}
\address{Dipartimento di Matematica\\ Universit\`{a} di Bologna\\ 40126~Bologna}
\email{francesco.milizia@unibo.it}

\author[Petrosyan]{Nansen Petrosyan}
\address{School of Mathematics, University of Southampton, Southampton, UK}\email{n.petrosyan@soton.ac.uk}

\author[Sisto]{Alessandro Sisto}
\address{Department of Mathematics, Heriot-Watt University, Edinburgh, UK}
	\email{a.sisto@hw.ac.uk}

\author[Vankov]{Vladimir Vankov}
\address{School of Mathematics, University of Bristol, Bristol, UK}
\email{vlad.vankov@bristol.ac.uk}

\begin{document}

\begin{abstract}
    For any geodesic metric space $X$, we give a complete cohomological characterisation of the hyperbolicity of $X$ in terms of vanishing of its second $\ell^{\infty}$-cohomology. We extend this result to the relative setting of $X$ with a collection of uniformly hyperbolic subgraphs. As an application, we give a cohomological characterisation of acylindrical hyperbolicity.
\end{abstract}

\maketitle

\section{Introduction}

Gersten proved that hyperbolic groups can be characterised, among finitely presented groups, as those with vanishing second $\ell^{\infty}$-cohomology \cite{Ger96}. Since then, many authors have characterised hyperbolic and relatively hyperbolic groups in terms of homological algebra (for example, \cite{AllcockGer}, \cite{Fra2018}, \cite{GH2009}, \cite{Min02}). A common thread amidst these is that local finiteness of corresponding spaces being acted on is a key condition. Our main result is a characterisation of hyperbolicity for general geodesic metric spaces, not just groups. The cohomology theory we use for this characterisation is built upon the ideas from \cite{Elek98}.

\begin{thm}\label{thm:intro1}
    Let $X$ be a geodesic metric space. Then,  the following are equivalent:
 \begin{enumerate}
    \item $X$ is hyperbolic;
    \item $X$  has a finite homological isoperimetric function and $\infcohom{X}{V}=0$ for all injective Banach spaces $V$;
    \item $X$  has a finite homological isoperimetric function and $\infcohom{X}{\ell^{\infty}(\mathbb N, \R)}=0$.
    \end{enumerate}
\end{thm}

When $X$ is a graph, admitting a finite isoperimetric function does not necessitate finite vertex degrees, as can be seen in the case of a Cayley graph with infinite generating set which we will encounter in Section \ref{sec:hypembed}. Such locally infinite graphs are pertinent to the definition of hyperbolically embedded subgroups, relevant for acylindrical hyperbolicity. Note that $\ell^\infty$-cohomology is invariant under quasi-isometry (Lemma \ref{lemma:quasi_isometry}, see also \cite{Elek98}), while it is currently an open question whether acylindrical hyperbolicity is a quasi-isometry invariant for finitely generated groups.

By extending our cohomology theory to the relative setting and proving an excision theorem, we give criteria for hyperbolicity of cusped spaces. As an application, we give the following characterisation of hyperbolically embedded subgroups, under the fairly mild additional condition that the ambient group $G$ has $S$-bounded $H^2$, see Definition \ref{defn:S-bounded}. We do not define this notion here in the introduction, but we point out that it is automatically satisfied if $G$ is finitely presented and $S$ is a finite generating set.

\begin{thm}
\label{thm:intro2}
    Let $G$ be a group with (possibly infinite) generating set $S$, and suppose that $G$ has $S$-bounded $H^2$. Let $H_1,\dots, H_n<G$ be finitely generated subgroups such that $d_S|_{H_i}$ is proper for every $i$, where $d_S$ is the word metric given by $S$.
    Then the family of subgroups $\{H_i\}$ is hyperbolically embedded in $(G,S)$ if and only if
    $$\infcohom{\cay{G,S},G/\{H_i\}}{V}=0$$
    for all injective Banach spaces $V$.
\end{thm}

Note that a group is acylindrically hyperbolic if and only if it contains a non-degenerate hyperbolically embedded subgroup, so the above provides a cohomological characterisation of acylindrical hyperbolicity (see \cite[Theorem 3.4]{Osin18}).

\subsection{A few words on the proof of Theorem \ref{thm:intro1}} \label{strategy} In \cite{Ger96}, Gersten was the first to prove a result of this kind, establishing a cohomological characterisation of hyperbolic groups among finitely presented groups $G$. He works with the universal cover $X$ of the presentation complex. The cocompactness of the action $G$ on $X$ is essential for Gersten's original definition of $\ell^{\infty}$-cohomology. He uses this property both to relate the vanishing of  $\infcohom{X}{\ell^{\infty}(\mathbb N, \R)}$ to the linearity of the homological Dehn function of $G$ and to connect the latter property to the hyperbolicity of $G$.

However, Gersten's method does not work in our generality due to the lack of finiteness assumptions on $X$. To address this,  we introduce a cohomology theory which is a natural generalisation of $\ell^{\infty}$-cohomology for general graphs. Its key feature is that it allows us to analyse $X$ in a filtered way by considering an ascending sequence of interconnected chain complexes on $X$. This approach provides sufficient control over the homological filling functions of $X$.

Furthermore, we need a generalisation  of \cite[Chapter III.H, Theorem 2.9]{BH13} to our homological setting. This is done in Proposition \ref{prop:linear_isop_hyp}, which states that if $X$ satisfies a homological linear isoperimetric inequality, then $X$ is hyperbolic.  To prove this, we  first utilise the fact that when $X$ is not hyperbolic,  it contains arbitrarily thick triangles, as in the proof of \cite[Chapter III.H, Theorem 2.9]{BH13}. The main challenge in our setting is that we are not dealing with curves but with 1-chains as dictated by our (co)homology theory. We resolve this problem by using a series of technical results which relate  the $\ell_1$-norm of 1-chains to the length of their support.

\subsection{Structure of the paper} In Section \ref{sec:hyperbolic_graphs}, we develop the theory of $\ell^{\infty}$-cohomology for general graphs and derive hyperbolicity criteria. In Section \ref{sec:relative}, we consider relative cohomology and prove an excision theorem, with application to cusped spaces. In Section \ref{sec:extending}, we prove an extension result for cocycles, necessary for a technical debt from Section \ref{sec:relative}. In Section \ref{sec:hypembed}, we derive criteria for a set of subgroups to be hyperbolically embedded.

\subsection*{Acknowledgements} All authors would like to thank the Heilbronn Institute for funding the focused workshop where this paper was started. 
The fourth author was supported by the Additional Funding Programme for Mathematical Sciences, delivered by EPSRC (EP/V521917/1) and the Heilbronn Institute for Mathematical Research.

\section{Cohomological characterisation of hyperbolic graphs}
\label{sec:hyperbolic_graphs}
We start this section by recalling the definition of $\ell^\infty$-cohomology of graphs, as introduced by Elek \cite{Elek98}.
In Elek's paper, additional hypotheses are required on the graph under consideration, but those are not needed to state the definition.
We then state and prove our cohomological characterisation of hyperbolic graphs, Theorem \ref{thm:hyp_main}, which is equivalent to Theorem \ref{thm:intro1}, as we discuss in the remark below.

\begin{remark}
    Any geodesic metric space is quasi-isometric to a connected graph. It will be convenient to work with graphs, and define the cohomology in that setting, and then one can define the $\ell^\infty$-cohomology of a geodesic metric space as the $\ell^\infty$-cohomology of any graph quasi-isometric to it. This is well defined because of quasi-isometry invariance of $\ell^\infty$-cohomology, which we prove as Lemma \ref{lemma:quasi_isometry}, see also \cite{Elek98}. Alternatively, one can also define $\ell^\infty$-cohomology as below using $X$ itself rather than the vertex set $\vertx{X}$.

Similarly to this, to make sense of the statement of Theorem \ref{thm:intro1}, we say that a geodesic metric space has finite homological isoperimetric function, as defined below, if some graph quasi-isometric to it does.
\end{remark}

We consider arbitrary undirected graphs.
More precisely, for us a graph $X$ is a $1$-dimensional CW-complex.
We denote by $\vertx{X}$ the set of vertices (i.e., $0$-cells) of a graph $X$.
We endow $X$ with a metric $d$ in which every edge (i.e., $1$-cell) has length $1$ (when the graph is not connected this is a metric with values in $\mathbb{R}_{\ge 0} \cup \{+\infty\}$).

For every $i \ge 0$, we denote by $C_i(X)$ the $\mathbb R$-vector space with basis $\mathcal{V}_X^{i+1}$.
For every natural number $R$, we define $C_i^R(X)$ as the subspace with basis
\[\{(x_0,\dots,x_{i})\in \vertx{X}^{i+1}: d(x_j,x_k)\leq R \ \forall j,k\}.\]
That is, $C_i^R(X)$ is spanned by tuples of diameter at most $R$.
We will always endow $C_i(X)$ and $C_i^R(X)$ with their $\ell^1$-norms with respect to the basis considered above; we denote these norms by $\norm{\cdot}_1$.

We denote by $\partial\colon C_i(X) \to C_{i-1}(X)$ the usual boundary operator, and set $B_i^R(X)=\partial C_{i+1}^R(X)$.
On $B^R_i(X)$ we consider the filling norm $\norm{\cdot}^R_F$, coming from the $\ell^1$-norm on $C_{i+1}^R(X)$, obtained by considering all possible fillings and taking the infimum of their norms. Formally, it is defined as follows:
\[ \norm{b}^R_F = \inf\{\norm{c}_1 : c \in C_{i+1}^R(X),\ \partial c = b\}.\]

\begin{remark}\label{rmk:ZisB}
    Let $b \in C_i(X)$ for some $i \ge 1$.
    If $\bd b = 0$, then $b = \bd c$ for some $c \in C_{i+1}(X)$ (this can be seen by considering a ``cone'' over $b$).
    That is, cycles in $C_i(X)$ are boundaries.
\end{remark}

For linear functions $f$ with values in a normed vector space $V$, defined on $C_i^R(X)$ or $B_i^R(X)$, we denote their operator norm by $\lvert f \rvert_\infty^R$ or $\lvert f\rvert_F^R$, respectively; they take values in $\mathbb{R}_{\ge 0} \cup \{+\infty\}$.
If $f$ is defined on a subspace of $C_i(X)$ containing $C_i^R(X)$, then $\lvert f \rvert_\infty^R$ denotes the norm of its restriction to $C_i^R(X)$.

We define
\[C_\infsub^i(X;V)=\{f\colon C_i(X)\to V\ :\  f|_{C_i^R(X)}\text{\ is\ bounded\ }\forall R\in\mathbb{N}\}.\]
For clarity, $f$ is linear, and boundedness is with respect to the $\ell^1$-norm (and the bound is not uniform over all $R$).
That is, $\lvert f \rvert_\infty^R < +\infty$ for every $R\in\mathbb{N}$.
We then define cocycles and coboundaries in the usual way, and obtain the $\ell^\infty$-cohomology of $X$ with coefficients in $V$, which we denote by $H^\bullet_\infsub(X;V)$.

\subsection{Paths, fillings and isoperimetric functions in graphs}
An oriented edge in a graph $X$ is an open $1$-cell endowed with an orientation; an oriented edge naturally defines a pair $(v,w) \in \vertx{X}^2$, where $v$ and $w$ are respectively the tail and the head of the oriented edge.

A path $p$ is a finite sequence of oriented edges $e_1,\dots,e_l$ such that, for $i = 1,\dots,l-1$, the head of $e_i$ is equal to the tail of $e_{i+1}$; the positive integer $l$ is the length of the path, and we denote it by $\length{p}$.
The tail of $e_1$ and the head of $e_l$ are, respectively, the starting and ending points of $p$.
The path is \emph{closed} if its starting and ending points coincide.

Let $p$ be a closed path in  $X$, and let $R$ be a positive integer.

An \emph{$R$-filling} $(\Delta,\Phi)$ of $p$ consists of a triangulation $\Delta$ of the disc $D^2$ (in the sense of \cite[Definition 2.1 in Chapter III.H]{BH13}, i.e. the boundary of each $2$-cell consists of three distinct $1$-cells forming a loop) and a map $\Phi\colon \Delta\to X$ (not necessarily continuous) such that:
\begin{itemize}
    \item The restriction of $\Phi$ to the boundary of the triangulated disc describes the closed path $p$;
    \item The image of each 2-simplex of $\Delta$ has diameter at most $R$;
    \item $\Phi$ sends vertices to vertices. 
\end{itemize}
Denote by $\lvert \Delta \rvert$ the number of 2-simplices of $\Delta$.
The \emph{$R$-area} of $p$ is defined to be:
\[\rarea{R}{p}=\min \left\{ \vert \Delta \rvert \, : \, (\Delta,\Phi) \text{ an } R\text{-filling of } p\right\}.\]
If there is no $R$-filling, then set $\rarea{R}{p}=\infty$.

\begin{definition}\label{def:finte_isop_function}
Let $X$ be a graph.
$X$ is said to have a \emph{finite isoperimetric function} if there exist an integer $R_0 \ge 1$ and a function $\theta\colon \mathbb{N} \to \mathbb{R}_{\ge 0}$ such that $\mathrm{Area}_{R_0}(p) \le \theta(\length{p})$ for every closed path $p$ in $X$.
\end{definition}
Notice that, for any closed path $p$, the quantity $\mathrm{Area}_R(p)$ is non-increasing as a function of $R$.
In particular, if $X$ has a finite isoperimetric function $\theta$ with respect to the $R_0$-area (as in \Cref{def:finte_isop_function}), then every $R \ge R_0$ works as well.

Every oriented edge with tail-head pair $(v,w)\in\vertx{X}^2$ defines an element of $C_1^1(X)$, given by the pair $(v,w)$ itself.
If $p$ is a path, we obtain an element of $C_1^1(X)$ by summing the pairs corresponding to the oriented edges appearing in the path.
We denote this element by $\chain{p}$.

Notice that, if $p$ is a closed path, then $\chain{p} \in C_1^1(X) \cap B_1(X)$.
On the other hand, as we see in the following lemma, every element of $C_1^1(X) \cap B_1(X)$ is a finite combination of closed paths and elements of the form $(x,x)$.
The statement is more general, as it provides a decomposition into paths of any element of $C_1^1(X)$, not necessarily boundaries, but of course in the general case also non-closed paths have to be used.

\begin{lemma}\label{lemma:sum_path}
    Let $X$ be a graph and let $T$ be a subset of $\vertx{X}$.
    Let $c \in C_1^1(X)$, with $\partial c$ supported on $T$.
    Then, $c$ can be written as a finite sum of the form
    \[ c = \sum_i\alpha_i \chain{p_i} + \sum_j\beta_j \chain{q_j} + \sum_k \mu_k\cdot(x_k,x_k) + \sum_l \nu_l\cdot [(y_l,z_l)+(z_l,y_l)],\]
    where $\alpha_i,\beta_j,\mu_k,\nu_l \in \mathbb{R}$, the $p_i$ are paths with endpoints in $T$, the $q_j$ are closed paths, the $x_k$ are vertices, and the $(y_l,z_l)$ are pairs of adjacent vertices, such that
    \[ \norm{c}_1 =  \sum_i\norm{\alpha_i}\cdot\length{p_i} + \sum_j\norm{\beta_i}\cdot \length{q_i} + \sum_k\norm{\mu_k} \]
    and $\sum_l\norm{\nu_l} \le \norm{c}_1$.
    If $c$ has coefficients in $\mathbb{Z}$, then there is such an expression with $\alpha_i,\beta_j,\mu_k, \nu_l \in \mathbb{Z}$.
\end{lemma}
\begin{proof}
    This can be deduced easily, e.g. from \cite[Theorem 6]{Min02}, where the analogous statement for cellular cycles is stated. 
    \end{proof}
    
    In the above, the terms of the form $\nu_l [(y_l,z_l)+(z_l,y_l)]$ are needed because some edges might need to be ``reversed'' before being incorporated into paths.
    As an example, suppose that $x,y,z$ are three pairwise-adjacent vertices, and take $c = (x,y)+(y,z)-(x,z)$.
    Notice that $\partial c = 0$.
    Then, we can write $c = \chain{q}-[(x,z)+(z,x)]$, where $q$ is the closed path that visits, in order, $x, y, z$ and returns to $x$.

Notice that $(x,z) + (z,x) = \chain{q_{xz}}$, where $q_{xz}$ is a closed path of length $2$ that visits $x$, $z$ and returns to $x$.
Therefore, in the statement of \Cref{lemma:sum_path} one could treat the terms of the form $\nu_l [(y_l,z_l)+(z_l,y_l)]$ as instances of elements of the form $\beta_j\chain{q_j}$; by doing this, thus writing only 
\[ c = \sum_i\alpha_i \chain{p_i} + \sum_j\beta_j \chain{q_j} + \sum_k \mu_k\cdot(x_k,x_k),\]
one only gets the worse estimate
\[ 3\norm{c}_1 \ge  \sum_i\norm{\alpha_i}\cdot\length{p_i} + \sum_j\norm{\beta_i}\cdot \length{q_i} + \sum_k\norm{\mu_k}. \]

\Cref{lemma:sum_path} applies only to elements of $C_1^1(X)$.
However, any $c \in C_1(X)$ is homologous to an element of $C_1^1(X)$, as in the following lemma.

\begin{lemma}\label{lemma:Rto1}
    Let $X$ be a graph and $R$ be a positive integer.
    Then, any $c \in C_1^R(X)$ can be written as
    \[ c = b + c',\]
    where $b \in B_1^R(X)$ and $c' \in C_1^1(X)$,
    with $\norm{c'}_1 \le R\cdot\norm{c}_1$, $\norm{b}_1 \le (R+1)\cdot\norm{c}_1$ and $\norm{b}_F^R \le (R+1)\cdot\norm{c}_1$.
    If $c$ has integer coefficients, then also $b$ and $c'$ can be taken with integer coefficients.
\end{lemma}
\begin{proof}
    Write $c = \sum_{i=1}^k\alpha_i(x_i,y_i)$ with $\alpha_i \in \mathbb{R}$, $(x_i,y_i)\in\vertx{X}^2$ and $d(x_i,y_i)\le R$.
    Let $p_i$ be a path of minimal length from $x_i$ to $y_i$, and suppose it encounters, in order, the vertices $p_i^0, \dots, p_i^\length{p_i}$, with $p_i^0 = x_i$ and $p_i^\length{p_i} = y_i$.
    The following identity holds:
    \[ \chain{p_i} - (x_i,y_i) = \bd \left[ -(p_i^0,p_i^0,p_i^0) + \sum_{j=1}^\length{p} (p_i^0, p_i^{j-1}, p_i^j) \right]. \]
    The $2$-chain in square brackets has $\ell^1$-norm at most $R+1$, and belongs to $C_2^R(X)$.
    Define $c' = \sum_{i=1}^k\alpha_i\chain{{p}_i}$; we have $c-c' \in B_1^R(X)$, with $\norm{c - c'}_F^R \le (R+1)\cdot\norm{c}_1$. 
    Moreover, $\norm{c'}_1 \le R\cdot\norm{c}_1$.
\end{proof}

With the next lemma, we can pass from a finite isoperimetric function in the sense of \Cref{def:finte_isop_function} to an analogous notion, where instead of fillings we consider $2$-chains.

\begin{definition}\label{def:fhif}
    Let $X$ be a graph.
    We say that $X$ has a \emph{finite homological isoperimetric function} if there exist a positive integer $R_0$ and a function $\theta\colon\mathbb N \to \mathbb{R}_{\ge 0}$ such that, for every closed path $p$, we have $\chain{p} \in B_1^{R_0}(X)$ and $\norm{\chain{p}}_F^{R_0} \le \theta(\length{p})$.
    
\end{definition}

\begin{lemma}\label{lemma:fif_fhif}
    Let $X$ be a graph with a finite isoperimetric function.
    Then, $X$ has a finite homological isoperimetric function (with the same $R_0$ and $\theta$).
\end{lemma}
\begin{proof}
The assumption of $X$ having a finite isoperimetric function directly implies the existence of $R_0 \ge 1$ and $\theta\colon\mathbb{N} \to \mathbb{R}_{\ge 0}$ such that 
$\rarea{R_0}{p} \le \theta(\length{p})$ for every closed path $p$.

An $R_0$-filling of a path $p$ naturally gives rise to an element of $C_2^{R_0}(X)$ whose boundary equals $\chain{p}$ and whose $\ell^1$-norm is at most the area of the filling.
It follows that $\chain{p} \in B_1^{R_0}(X)$ and
\[\norm{\chain{p}}_F^{R_0} \le \rarea{R_0}{p} \le \theta(\length{p}),\]
proving the assertion.
\end{proof}

In the following lemma we use Lemma \ref{lemma:sum_path} to upgrade a finite homological isoperimetric inequality to a filling inequality.

\begin{lemma}\label{lemma:fhif}
    Let $X$ be a graph with a finite homological isoperimetric function, with parameters $R_0$ and $\theta_0$.
    Then there exists a function $\theta\colon\mathbb N \to \mathbb{R}_{\ge 0}$ such that the following hold for every $R \ge R_0$:
    \begin{enumerate}
        \item\label{it:fif_fill_R} If $b \in C_1^R(X) \cap B_1(X)$ then $b \in B_1^{R}(X)$;
        \item\label{it:fif_fill_norm} If $b \in B_1^{R}(X)$ has integer coefficients, then $\lvert b \rvert_F^{R} \le \theta(R \cdot \lvert b\rvert_1)$.
    \end{enumerate}
    If $\theta_0$ is super-additive, then we can take $\theta(s)=\theta_0(3s)+5s$.
\end{lemma}

\begin{proof}
Let $R_0\ge 1$ and $\theta_0\colon\mathbb{N}\to\mathbb{R}_{\ge 0}$ be the positive integer and the function as in \Cref{def:fhif}.
We can assume that $\theta_0$ is super-additive, that is, $\theta_0(r+s) \ge \theta_0(r) + \theta_0(s)$ for every $r,s\in\mathbb{N}$.

By construction, for every $R\ge R_0$ and closed path $p$, we have $\chain{p} \in B_1^R(X)$ and
\[\lvert \chain{p} \rvert_F^R \le \lvert \chain{p} \rvert_F^{R_0} \le \theta_0(\length{p}).\]

If $b \in C_1^1(X) \cap B_1(X)$, then we can write
\[b = \beta_1 \chain{q_1} + \dots + \beta_m \chain{q_m} + \mu_1(x_1,x_1) + \dots + \mu_n(x_n,x_n)\]
as in \Cref{lemma:sum_path} (with $T$ empty, and treating the elements of type $\nu_l [(y_l,z_l)+(z_l,y_l)]$ as chains associated to closed paths), where the $q_j$ are closed paths.
Notice that $(x_j,x_j) = \partial (x_j,x_j,x_j)$, so it belongs to $B_1^1(X)$.
This implies that $b$ belongs to $B_1^{R}(X)$ for every $R \ge R_0$, since every summand does.
If, moreover, $b$ has integer coefficients, we can assume that $\beta_i,\mu_j \in \{-1,+1\}$, and we have
\[\length{q_1} + \cdots + \length{q_m} + n \le 3\norm{b}_1.\]
Therefore, for every $R \ge R_0$,
\begin{align*}
    \lvert b \rvert_F^{R} &\le \theta_0(\length{q_1}) + \dots + \theta_0(\length{q_m}) + n \\
    &\le \theta_0(\length{q_1} + \cdots + \length{q_m}) + n\\
    &\le \theta_0(3\lvert b\rvert_1) + 3\lvert b \rvert_1.
\end{align*}

We now prove \eqref{it:fif_fill_R} and \eqref{it:fif_fill_norm} without assuming $b \in C_1^1(X)$.
Fix $R \ge R_0$ and suppose that  $b\in C_1^R(X) \cap B_1(X)$.
From \Cref{lemma:Rto1}, we know that $b$ is homologous to some $b' \in C_1^1(X)$, with $\norm{b'}_1 \le R\cdot\norm{b}_1$, so that $\lvert b - b'\rvert_F^R \le (R+1)\cdot\lvert b\rvert_1$. 
In particular, $b' = b - (b-b')$ belongs to $B_1(X)$.
Since $b' \in C_1^1(X) \cap B_1(X)$, we already know that $b' \in B_1^R(X)$. Therefore, $b = (b-b') + b' \in B_1^R(X)$, proving \eqref{it:fif_fill_R}.
    
If $b$ has integer coefficients, then we can assume that also $b'$ has integer coefficients, and since it belongs to $C_1^1(X) \cap B_1(X)$ we conclude that
\begin{align*}
        \lvert b \rvert_F^R &\le \lvert b' \rvert_F^R + \lvert b - b' \rvert_F^R \\
        &\le \theta_0(3\lvert b' \rvert_1) + 3\lvert b' \rvert_1 
 + (R+1)\cdot\lvert b \rvert_1\\
        &\le \theta_0(3R\cdot \lvert b\rvert_1) + (4R+1)\cdot\lvert b \rvert_1.
    \end{align*}
    Therefore, we get \eqref{it:fif_fill_norm} by setting, for instance, $\theta(s) = \theta_0(3s) + 5s$. 

\end{proof}

\subsection{The statement}
The following is the main result of this section.
\begin{thm}\label{thm:hyp_main}
    Let $X$ be a connected graph. The following conditions are equivalent:
 \begin{enumerate}
    \item\label{item_hyp_main_hyp} $X$ is hyperbolic;
    \item\label{item_hyp_main_V} $X$ has a finite homological isoperimetric function and $\infcohom{X}{V}=0$ for all 1-injective Banach spaces $V$;
    \item\label{item_hyp_main_univ} $X$ has a finite homological isoperimetric function and $\infcohom{X}{\ell^{\infty}(\mathbb N, \R)}=0$.
    \end{enumerate}
\end{thm}

Recall that a Banach space $V$ is 1-injective if, for every linear function $f\colon U\to V$ defined on a subspace $U$ of a normed vector space $W$, there exists a linear function $F\colon W\to V$ that extends $f$, with $\lvert F \rvert \le \lvert f \rvert$.
Typical examples are $V = \mathbb{R}$ with the usual Euclidean norm (by the Hahn-Banach theorem), and more generally spaces $V = \ell^\infty(S,\mathbb{R})$ of bounded functions from a set $S$ to the real numbers, endowed with the sup norm.

The rest of the section is devoted to proving the implications \eqref{item_hyp_main_hyp} $\implies$ \eqref{item_hyp_main_V} (\Cref{prop:hyp_vanishing}) and \eqref{item_hyp_main_univ} $\implies$ \eqref{item_hyp_main_hyp} (\Cref{prop:vanishing_hyp}), the remaining implication being straightforward.

\subsection{From hyperbolicity to cohomology}

We start by showing that hyperbolic graphs have vanishing $\ell^\infty$-cohomology:

\begin{prop}\label{prop:hyp_vanishing}
    Let $X$ be a connected graph.
    If $X$ is hyperbolic, then $\infcohom{X}{V}=0$ for all $1$-injective Banach spaces $V$.
    More precisely, for all $\delta \ge 0$, there exist an integer $R_0 \ge 0$ and a function $K\colon\mathbb{N}\times\mathbb R_{\ge 0}\to\mathbb R_{\ge 0}$ such that, if $X$ is $\delta$-hyperbolic, $V$ is $1$-injective, and $f\in Z^2_\infsub(X;V)$, then there exists $g\in C_\infsub^1(X;V)$ with $\cbd g=f$ and $|g|^R_\infty\leq K(R,|f|_\infty^R)$ for every integer $R \ge R_0$.
\end{prop}

The key to proving the proposition is the homological isoperimetric inequality given in the following lemma. 

\begin{lemma}\label{lemma:hyp_filling}
    For every $\delta \ge 0$ there exist $R_0 \in \mathbb{N}_{\ge 1}$ and $C \in \mathbb{R}_{\ge 0}$ such that the following hold for every $\delta$-hyperbolic graph $X$ and every $R \ge R_0$:
    \begin{itemize}
        \item If $b \in C_1^R(X) \cap B_1(X)$, then $b \in B_1^R(X)$;
        \item If $b \in B_1^{R}(X)$, then $\lvert b \rvert_F^{R} \le C\cdot R \cdot \lvert b \rvert_1$.
    \end{itemize}
\end{lemma}

\begin{proof}
If $X$ is a $\delta$-hyperbolic graph, then $X$ has a \emph{linear} isoperimetric function \cite[Chapter III.H, Proposition 2.7]{BH13}, meaning that there exist $R_0 \ge 1$ and $\theta_0\colon\mathbb{N} \to \mathbb{R}$ as in \Cref{def:finte_isop_function}, with $\theta_0$ of the form $\theta_0(s) = C_0 \cdot s$ for some $C_0 \in \mathbb{R}_{\ge 0}$.
The numbers $R_0$ and $C_0$ can be chosen uniformly for all $\delta$-hyperbolic graphs, i.e. they depend only on $\delta$.

We now apply Lemmas \ref{lemma:fif_fhif} and \ref{lemma:fhif}; note that \eqref{it:fif_fill_R} and \eqref{it:fif_fill_norm} of \Cref{lemma:fhif} hold (for every $R \ge R_0$) with the function  $\theta$ given by $\theta(s) = \theta_0(3s)+5s = (3C_0+5) \cdot s$.
We set $C = 3C_0+5$.

For the inequality $\lvert b \rvert_F^R \le C\cdot R\cdot\lvert b\rvert_1$ to hold, which is \eqref{it:fif_fill_norm} from \Cref{lemma:fhif}, we do not need to assume that $b$ has integer coefficients.
In fact, in the proof of \Cref{lemma:fhif}, we needed the coefficients to be integers only to deduce that 
\[ \sum_{i=1}^m \lvert\beta_i\rvert \cdot \theta_0(\length{q_i}) \le \theta_0\left(\sum_{i=1}^m \lvert\beta_i\rvert \cdot\length{q_i}\right),\]
where $\beta_i$ and $q_i$ are as in \Cref{lemma:sum_path}, using the super-additivity of $\theta_0$.
But if $\theta_0(s) = C_0\cdot s$, this inequality holds regardless of the fact that the $\beta_i$ are integers.
\end{proof}

We are now ready to prove Proposition \ref{prop:hyp_vanishing}.

\begin{proof}[Proof of Proposition \ref{prop:hyp_vanishing}]
    Given a 2-cocycle $f \in Z_\infsub^2(X;V)$, define a linear map $h\colon B_1(X)\to V$, setting $h(b) = f(c)$ for any $c \in C_2(X;V)$ with $\bd c = b$.
    This is well defined: by \Cref{rmk:ZisB} any two such $c$ differ by a boundary, and $f$ vanishes on boundaries.
    
    Let $R$ be any natural number. If $b \in B_1^R(X)$, then for any $\varepsilon > 0$ we can write $b = \bd c$ for some $c \in C_1^R(X)$ with $\lvert c \rvert_1 \le \lvert b\rvert_F^R + \varepsilon$, which implies $\lVert h(b) \rVert = \lVert f(c)\rVert \le \lvert f\rvert_\infty^R \cdot \lvert c \rvert_1 \le \lvert f\rvert_\infty^R \cdot (\lvert b\rvert_F^R+\varepsilon)$; by letting $\varepsilon \to 0$ we get $\lVert h(b) \rVert \le \lvert f\rvert_\infty^R \cdot \lvert b\rvert_F^R$.
    In particular, for every $R$ the restriction of $h$ to $B_1^R(X)$ is bounded with respect to the $\lvert\cdot\rvert_F^R$ norm. 
    We wish to extend the map $h$ to $C_1(X)$, in such a way that for every $R$ the restriction to $C_1^R(X)$ is bounded with respect to the $\ell^1$-norm.
    
    Fixing the appropriate $R_0 \in \mathbb{N}$, we can apply the linear isoperimetric inequality of \Cref{lemma:hyp_filling}, which yields $\lvert\cdot\rvert_F^{R}\leq C\cdot R\cdot \lvert\cdot\rvert_1$ on $B_1^{R}(X)$ for every $R \ge R_0$.
    In particular, taking $R = R_0$, we obtain that $h$ is bounded with respect to the $\ell^1$-norm on $B_1^{R_0}(X)$.
    Since $V$ is $1$-injective, we can extend $h\colon B_1^{R_0}(X) \to V$ to a linear function $g:C^{R_0}_1(X)\to V$, with norm bounded above by $C\cdot R_0 \cdot \lvert f \rvert_\infty^{R_0}$.
    
    The linear extension of $g$ to the whole $C_1(X)$ is now forced, if we require that $g = h$ on $B_1(X)$.
    In fact, if $a,b \in \vertx{X}$ and there is a path $p$ in $X$, passing through vertices $p_0, \dots, p_m$, where $p_0 = a$ and $p_m = b$, we are forced to have
    \begin{equation}\label{eq:def_g}
        g(a, b) = \sum_{i=1}^m g(p_{i-1},p_i) - \sum_{i=1}^{m-1}f(p_0,p_i,p_{i+1}),
    \end{equation}
    because it must hold that $\cbd g = f$ (this condition is equivalent to $g$ being an extension of $h$).
    We already know, by the way we constructed $g$ on $C_1^{R_0}(X)$, that this formula holds whenever $a$ and $b$ are at distance $\le R_0$ apart; once we have proved that the value of the right-hand side does not depend on the path, we will actually have extended $g$ to the whole $C_1(X)$, by using the right-hand side of \eqref{eq:def_g} to \emph{define} $g(a,b)$.
    To simplify the notation, given a path $p$ as above, we set $\widehat{p} = \sum_{i=1}^{m-1}(p_0,p_i,p_{i+1}) \in C_2(X)$, so that the formula reads as $g(a,b) = g(\chain{p}) - f(\widehat{p})$.
    Let $q$ be another path, passing through vertices $q_0, \ldots, q_n$, going from $q_0 = a$ to $q_n = b$.
    Then, by \Cref{lemma:hyp_filling}, we have that $\chain{p}-\chain{q} \in B_1^{R_0}(X)$, so we can take $c \in C_2^{R_0}(X)$ with $\bd c = \chain{p}-\chain{q}$.
    Notice that also $\widehat{p}-\widehat{q}$ has boundary $\chain{p}-\chain{q}$, so $f(c) = f(\widehat{p}-\widehat{q})$.
    Therefore,
    \[ g(\chain{p}) - g(\chain{q}) = g(\chain{p}-\chain{q}) = f(c) = f(\widehat{p}-\widehat{q}) = f(\widehat{p}) - f(\widehat{q}),\]
    and the independence on the path is proved.
    
    We now check that $\cbd g = f$.
    Given $(a,b,c) \in \vertx{X}^3$, we have to prove that $f(a,b,c) = g(a,b) + g(b,c) - g(a,c)$.
    Take paths $p$ and $q$, passing through vertices $p_0,\dots,p_m$ and $q_0,\dots,q_n$ respectively, with $p_0 = a$, $p_m = b = q_0$ and $q_n = c$.
    As a path joining $a$ to $c$ we consider the composition of $p$ and $q$.
    Using these paths to express the values attained by $g$, proving the equality $f(a,b,c) = g(a,b) + g(b,c) - g(a,c)$ amounts to showing (expanding each of the summands $g(a,b)$, $g(b,c)$ and $g(a,c)$ using expression \eqref{eq:def_g}) that
    \[ f(p_0,q_0,q_n) = f(p_0,q_0,q_1) + \sum_{i=1}^{n-1}f(p_0,q_i,q_{i+1}) - f(q_0,q_i,q_{i+1}).\]
    But this holds because the right-hand side is the evaluation of $f$ at a chain whose boundary is equal to $\bd(p_0,q_0,q_n)$.
    
    Finally, if $R \ge R_0$ and $a,b\in\vertx{X}$ are at mutual distance $\le R$, from the expression \eqref{eq:def_g}, taking a path of minimal length from $a$ to $b$, it follows that
    \begin{align*}
      \lVert g(a,b) \rVert &\le R \cdot C \cdot R_0 \cdot \lvert f \rvert_\infty^{R_0} + (R-1)\cdot \lvert f \rvert_\infty^R\\
      &\le \lvert f \rvert_\infty^R \cdot R \cdot (C \cdot R_0 + 1).
    \end{align*}
    For the first inequality we have used that the terms on the right-hand side of \eqref{eq:boundary} satisfy $\lVert g(p_{i-1},p_i)\rVert \le \lvert g\rvert_\infty^{R_0} \le C\cdot R_0 \cdot \lvert f\rvert_\infty^{R_0}$ and $\lVert f(p_0,p_i,p_{i+1})\rVert \le \lvert f \rvert_\infty^R$.
    Hence, $\lvert g \rvert_\infty^R \le \lvert f \rvert_\infty^R \cdot R \cdot (C \cdot R_0 + 1)$, so we can take $K(x,y)=(C\cdot R_0+1)xy$.
\end{proof}

\subsection{From cohomology to hyperbolicity}
In order to prove hyperbolicity starting from vanishing of $\ell^\infty$-cohomology, we have to show that a graph satisfying a homological linear isoperimetric inequality is hyperbolic.
The connection of this result with various other well-known results from the literature of this type is explained below.

\begin{definition}\label{def:hom_linear_isop_ineq} Let $X$ be a graph. We say that $X$ satisfies a {\it homological linear isoperimetric inequality} if there exists a positive integer $R_0$ with the following property:
for every $R\ge R_0$ there exists $K_R \in \mathbb{R}_{\ge 0}$ such that $\lvert b \rvert_F^{R} \le K_R \cdot \lvert b\rvert_1$ for every $b \in B_1^{R}(X)$.
    \end{definition}

\begin{prop}
\label{prop:linear_isop_hyp}
    Let $X$ be a connected graph.
    If $X$ satisfies a homological linear isoperimetric inequality, then $X$ is hyperbolic.
\end{prop}
The proof of Proposition \ref{prop:linear_isop_hyp} is similar to the proof of \cite[Proposition~4.2]{KK21} where a variation of a homological linear isoperimetric  inequality on a graph is shown to imply hyperbolicity.
This is an adaptation of the proof of the corresponding fact for the usual isoperimetric inequality in \cite[Chapter III.H, Theorem 2.9]{BH13}. 

We warn the reader that the norm considered in \cite{KK21} is not the standard $\ell^1$-norm on a normed vector space $V$, since it counts the cardinality of the support of an element in $V$, forgetting the weights. Since we need to deal with real chains, simply counting the cardinality of the support of an element does not give a lower bound on its $\ell^1$-norm.
We will use \Cref{lemma:sum_path} to remedy this problem.

\begin{proof}[Proof of Proposition \ref{prop:linear_isop_hyp}] By \Cref{lemma:fhif}, there exists an integer $R_0\ge 1$ so that $C_1^1(X) \cap B_1(X)\subseteq B_1^{R_0}(X)$. We may suppose that $X$ satisfies a homological linear isoperimetric inequality (\Cref{def:hom_linear_isop_ineq}) for $R \ge R_0$.
Fix any $R \ge R_0$.

Denote by $\kappa=K_R$ the isoperimetric constant and set $k=12\kappa R^2+1$, $m=3\kappa R$. For simplicity, we take $\kappa$  to be an integer.

First, we observe that, from the very definition of $\partial\colon C_2(X) \to C_1(X)$,
\begin{equation}\label{eq:boundary}
3\lvert a \rvert_1 \geq \lvert \partial a\rvert_1
\end{equation}
for any 2-chain $a\in C^R_2(X)$.

By a way of contradiction, suppose $X$ is not hyperbolic.
There must be an arbitrarily large  $n>6k$ and a geodesic triangle $\Delta\subseteq X$ which is not $(n+1)$-slim.
This means that there is a vertex $v$ on one of the sides which does not lie in the $n$-neighbourhood of the other two sides. 
In the first part of the proof of \cite[Chapter III.H, Theorem 2.9]{BH13}, it is shown that by either `cutting the corners' or `cutting a corner and the opposite edge' of $\Delta\subset X$, one obtains a geodesic polygon $H$ in $X$ (see Figure \ref{fig:hexagons}) with the properties listed below (we treat the two cases separately).\\

\begin{figure}[ht]
    \centering
    \includegraphics[scale=1]{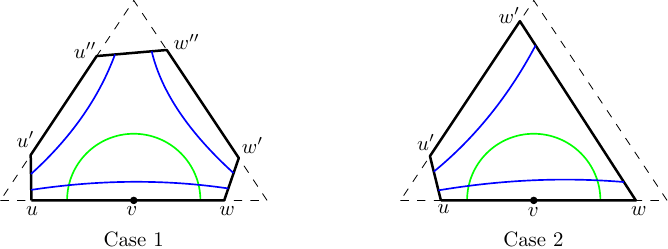}
    \caption{The triangle $\Delta$ and the polygon $H$.}
    \label{fig:hexagons}
\end{figure}

\noindent{}{\bf Case 1:} $H$ is a geodesic hexagon, with:
\begin{itemize}
    \item $d(u, u')=d(w,w')=d(u'',w'')=2k$;
    \item The $(k-1)$-neighbourhoods (blue in Figure \ref{fig:hexagons}) of the segments $[u,w]$, $[u',u'']$, and $[w',w'']$ are pairwise disjoint;
    \item $\exists v\in [u,w]$ such that the $(n-k+1)$-neighbourhood of $v$ (green in Figure \ref{fig:hexagons}) is disjoint from the $(k-1)$-neighbourhoods of $[u',u'']$, and $[w',w'']$.
\end{itemize}
Let $\alpha$, $\beta$, and $\gamma$ denote the lengths of $[u,w]$, $[u',u'']$, and $[w',w'']$, respectively.

Putting weight $1$ on every edge of  $H$, with the orientation obtained by travelling along $H$ so that $w$, $v$ and $u$ are visited in this order (clockwise in \Cref{fig:hexagons}), we get a cycle $h\in C^1_1(X)$.
By definition of the filling norm $\lvert \cdot \rvert_F^{R}$, there is a chain $f \in C^R_2(X)$ such that $\partial f=h$ and $\lvert f\rvert_1\leq \lvert h \rvert_F^{R}+1$. 

For $1\leq i\leq m$, let $u_i\in [u, u']$ and $w_i\in [w, w']$  be the vertices which are at distance $4Ri-2R$ from $u$ and $w$, respectively, and define the following sets of vertices:
\begin{align*}
A'_i =\ &\{x\in [u,u'] \; | \; 4R(i-1)\leq d(u, x) \leq 4Ri-2R \}\\
&\cup \{x\in [w,w'] \; | \; 4R(i-1)\leq d(w, x)\leq 4Ri-2R \},\\
A_i =\ &\{x\in [u,u'] \; | \; 4Ri-2R\leq d(u, x)\leq 4Ri \}\\
&\cup\{x\in [w,w'] \; | \; 4Ri-2R\leq d(w, x)\leq 4Ri \}.
\end{align*}

Denote by $B_i$ the neighbourhood of $[u,w]$ of radius $4Ri$ (i.e. the set of vertices at distance $\le 4Ri$ from vertices of $[u,w]$).
Since $4Ri\leq 4Rm=12\kappa R^2=k-1$, each $B_i$ is contained in the $(k-1)$-neighbourhood of $[u,w]$.

Let $f_i$ be the restriction of $f$ to the 2-simplices (triples of vertices) supported in $B_{i}\smallsetminus B_{i-1}$.
Since $\partial f=h$, it follows that $\partial f_i$ has weight 1 on the (head-tail pairs corresponding to) edges of $[u,u']\cup [w,w']$ incident to $u_i$ and $w_i$, because each such edge cannot be in the boundary of a 2-simplex which is \emph{not} contained in $B_{i}\smallsetminus B_{i-1}$.
Also, note that $\partial f_i$ splits as the sum of a 1-chain $\varphi'_i$ supported in 
$$\{R\mbox {-neighbourhood of } B_{i-1}\}\cup A'_i$$ 
and a 1-chain $\varphi_i$  supported in 
$$\{R\mbox {-neighbourhood of } \overline{X\smallsetminus B_{i}}\}\cup A_i,$$ 
 (see Figure \ref{fig:paths}).
 These two subsets of $\vertx{X}$ intersect at exactly $T:=\{u_{i}, w_{i}\}$, and by the above observation, 
$$\partial \varphi_i=w_{i}-u_{i}\;\; \mbox{ and } \;\; \partial \varphi'_i= u_{i}-w_{i}.$$

\begin{figure}[h] 
    \[
    \begin{tikzpicture}

\coordinate (A) at (-2, 0);          
\coordinate (B) at (2, 0);           
\coordinate (C) at (3, 2);          
\coordinate (D) at (1, 4);
\coordinate (E) at (-1, 4);
\coordinate (F) at (-3, 2);

        \draw [thick] (A) -- (B);
        \draw [thick] (C) -- (D);
        \draw [thick] (D) -- (E);
        \draw [thick] (E) -- (F);
        \draw [thick] (F) -- (-2.5,1);  
           \draw [thick] (A) -- (-2.1,.2);  
           \draw [thick, line width=1.2pt, red] (-2.5,1) -- (-2.3,.6);
            \draw [thick, line width=1.2pt, blue] (-2.1,.2) -- (-2.3,.6);
            \draw [thick, line width=1.2pt, red] (2.5,1) -- (2.3,.6);
            \draw [thick, line width=1.2pt, blue] (2.1,.2) -- (2.3,.6);
              \draw [thick] (B) -- (2.1,.2);           
         \draw [thick] (C)-- (2.5,1);           

\node [below left] at (A) {$u$};
\node [below right] at (B) {$w$};
\node [right] at (C) {$w'$};
\node [above right] at (D) {$w''$};
\node [above left] at (E) {$u''$};
\node [left] at (F) {$u'$};

\node [left] at (-2.3,.6) {$u_i$};

\node [right] at (2.3, .6) {$w_i$};

\node [left, blue] at (0.3,1.1) {$\varphi'_{i}$};
\node [left] at (0.5,.4) {$B_{i-1}$};
\node [left, red] at (0.3, 2) {$\varphi_{i}$};

\node at (-2.3,.6)  [circle,fill,inner sep=1pt]{};
\node at (2.3,.6)  [circle,fill,inner sep=1pt]{};

\draw [red, thick, line width=1.2pt, ] (2.5,1) to[bend right=30] (-2.5,1);
\draw [blue, thick,  line width=1.2pt, ] (2.1,.2) to[bend right=30] (-2.1,.2);

\end{tikzpicture}\]
\caption{The filling $f$ of $h$.}\label{fig:paths}   
\end{figure}
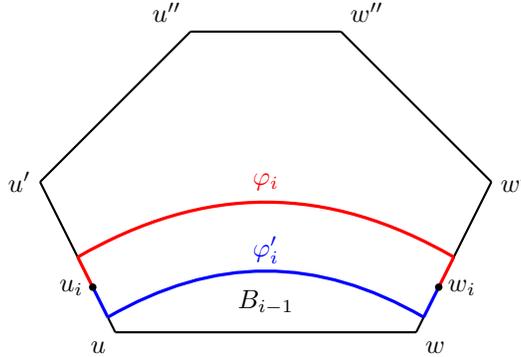

 By combining \Cref{lemma:sum_path} and \Cref{lemma:Rto1}, we get $\varphi_i=\sum_j\lambda_j\chain{p_j}\ + \ \sum\sigma_k$, where the summands $\sigma_k$ lie in $B_1(X)$, with \[\sum_j\lambda_j\cdot\length{p_j} \le R\cdot \norm{\varphi_i}_1,\]
and the paths $p_j$ have endpoints $u_i$ and $w_i$.
Since $\partial \varphi_i=w_{i}-u_{i}$, it follows that $\sum_j \norm{\lambda_j} \geq 1$.
Also, every $p_j$ must be of length at least $d(u_{i},w_{i})\geq \alpha-2k+2$.
Therefore, we obtain
  $$R \cdot \lvert \varphi_i\rvert_1 \ge \sum_j \norm{\lambda_j} \cdot\length{p_j}\geq \alpha-2k+2.$$
The same holds for $\lvert \varphi'_i\rvert_1$.
Hence, using (\ref{eq:boundary}), for $1\leq i\leq m$, we get
$$\lvert f_i \rvert_1 \geq \frac{1}{3}(\lvert \partial f_i\rvert_1) = \frac{1}{3}(\lvert \varphi_i\rvert_1+\lvert \varphi'_i\rvert_1) \geq \frac{2}{3R}(\alpha-2k+2).$$

Denote $f_{[u,w]}=\sum_{i=1}^{m} f_i$.
Since the 2-chains $f_i$ have disjoint supports, 
$$\lvert f_{[u,w]} \rvert_1=\sum_{i=1}^{m}\lvert f_i \rvert_1 \geq\frac{2m}{3R}(\alpha-2k+2)=2\kappa (\alpha-2k+2).$$

Similarly,
$$\lvert f_{[u', u'']} \rvert_1 \geq 2\kappa(\beta-2k+2),$$
$$\lvert f_{[w',w'']} \rvert_1 \geq 2\kappa(\gamma-2k+2).$$
Note that the 2-chains $f_{[u,w]}$ , $f_{[u',u'']}$, and $f_{[w',w'']}$ have pairwise disjoint supports,  since they are in $(k-1)$-neighbourhoods of the segments $[u,w]$, $[u',u'']$, and $[w',w'']$, respectively. Therefore, we have
$$\lvert f \rvert_1\geq 2\kappa(\alpha + \beta + \gamma -6k+6).$$
On the other hand, since $\lvert h \rvert_1=\alpha+\beta+\gamma +6k$, by the homological isoperimetric inequality and our choice of $p$, 
\[\kappa (\alpha+\beta+\gamma +6k)+1\geq \lvert f \rvert_1 \geq 2\kappa(\alpha + \beta + \gamma -6k+6).\]
This implies that $\alpha+\beta +\gamma\leq 18k$. Since $\alpha\geq 2n-4k$, this puts an upper bound on $n$ that only depends on $k$, leading to  a contradiction. 
\\

\noindent {\bf Case 2:} $H$ is a quadrilateral with
\begin{itemize}
    \item $d(u, u')=2k$, $d(w,w')=4k$;
    \item The $(k-1)$-neighbourhoods (blue in Figure \ref{fig:hexagons}) of the segments $[u,w]$ and $[u',w']$ are  disjoint;
    \item $\exists v\in [u,w]$ such the $(n-2k)$-neighbourhood of $v$ (green in Figure \ref{fig:hexagons}) is disjoint from the $(k-1)$-neighbourhoods of $[u', w']$.
\end{itemize}
Letting $\alpha=[u,w]$ and $\beta=[u',w']$ and arguing analogously, one obtains
$$\lvert f \rvert_1  \geq 2\kappa(\alpha + \beta  -4k+4).$$

Since $\lvert h \rvert_1 =\alpha +\beta + 6k$, the homological linear isoperimetric inequality leads again to a $k$-bound on $n$, which is a contradiction.
    \end{proof}

To shorten the notation, let us denote $\linfn :=\ell^{\infty}(\mathbb N, \R)$. We are now ready to complete our cohomological characterisation of hyperbolicity with the following proposition:

\begin{prop}
\label{prop:vanishing_hyp}
    Let $X$ be a connected graph having a finite homological isoperimetric function.
    If $\infcohom{X}{\linfn}=0$, then $X$ is hyperbolic.
\end{prop}

The proof involves combining the isoperimetric characterisation of hyperbolicity from Proposition \ref{prop:linear_isop_hyp} with the following lemmas. Roughly, in Lemma \ref{lemma:extend_BC} we would like to extend functionals on $B_1^R(X)$ to functionals on $C_1^R(X)$. In order to do so, we want to first extend to $B_1(X)$ and then use the vanishing of cohomology. Since we assume that the coefficient Banach space is 1-injective, to perform the first extension it turns out that we only need the control on norms given by the following lemma.

\begin{lemma}\label{lemma:fif_step}
    Let $X$ be a graph having a finite homological isoperimetric function.
    Then, there exists a positive integer $R_0$, and constants $K_R \in \mathbb{R}_{\ge 0}$, for $R \ge R_0$, such that, if $b \in B_1^R(X)$, then $\lvert b \rvert_F^{R} \le K_R \cdot \lvert b \rvert_F^{R+1}$.
\end{lemma}
\begin{proof}
    Let $R_0$ and $\theta\colon\mathbb N \to \mathbb{R}_{\ge 0}$ be as in \Cref{lemma:fhif}.
    Let $b \in B_1^R(X)$, with $R \ge R_0$, and let $c \in C_2^{R+1}(X)$ be a filling of $b$.
    By definition, $c$ is a linear combination of triples $(x_0,x_1,x_2) \in \vertx{X}^3$ of diameter at most $R+1$.
    Every such triple has three sides, $(x_0,x_1), (x_0,x_2)$ and $(x_1,x_2)$; we call a side \emph{long} if its two vertices are at distance $R+1$, otherwise we call it \emph{short}.
    Notice that $b = \bd c$ is a linear combination of the sides of the triples in $c$, and since $b \in C_1^R(X)$, all long sides get simplified.

    For every long side $(x,y)$, we fix an intermediate vertex $v_{xy} \in \vertx{X}$ with distance at most $R$ from both $x$ and $y$, and define the \emph{subdivision} of $(x,y)$ as $(x,y)^R = (x,v_{xy}) + (v_{xy},y) \in C_1^R(X)$.
    If $(x,y)$ is short, we simply set $(x,y)^R = (x,y)$.
    Notice that in both cases $\lvert(x,y)^R\rvert_1 \le 2$.

    Fix any $\varepsilon > 0$. We substitute every triple $(x_0,x_1,x_2)$ occurring in $c$ with a chain  $(x_0,x_1,x_2)^R \in C_2^R(X)$ with $\bd (x_0,x_1,x_2)^R = (x_0,x_1)^R + (x_1,x_2)^R - (x_0,x_2)^R$ and $\lvert (x_0,x_1,x_2)^R \rvert_1 \le \theta(6R) + \varepsilon$, where we are using property \eqref{it:fif_fill_norm} of \Cref{lemma:fhif} and the definition of the filling norm.
    This procedure gives as a result a new $2$-chain $c^R \in C_1^R(X)$ with $\bd c^R = b$ and $\lvert c^R \rvert_1 \le (\theta(6R) + \varepsilon)\lvert c\rvert_1$.
    
    Letting $\varepsilon \to 0$ and $\lvert c \rvert_1 \to \lvert b\rvert_F^{R+1}$, we obtain that $\lvert b\rvert_F^R \le \theta(6R)\cdot\lvert b\rvert_F^{R+1}$.
\end{proof}

\begin{lemma}\label{lemma:extend_BC}
    Let $X$ be a graph with a finite homological isoperimetric function.
    Then, there is a positive integer $R_0$ such that the following holds for any $1$-injective Banach space $V$ and every $R \ge R_0$:
    if $\infcohom{X}{V}=0$, then any bounded linear function $g\colon B_1^R(X) \to V$ can be extended to a bounded linear function on $C_1^R(X)$.
\end{lemma}
\begin{proof}
    Take $R_0$ as in \Cref{lemma:fif_step}: if $R \ge R_0$, we have $\lvert\cdot\rvert^R_F\leq K_R \lvert\cdot\rvert^{R+1}_F$ as norms on $B_1^R(X)$.
    This implies that $g\colon B_1^R(X) \to V$ can be extended to $B_1^{R+1}(X)$, with $\lvert g \rvert_F^{R+1} \le K_R \cdot \lvert g \rvert_F^R$.
    Performing these extensions inductively, we get an extension on $B_1(X)$, with $\lvert g \rvert_F^R < +\infty$ for every $R \ge R_0$.

    Now define $f\colon C_2(X) \to V$, setting $f(c) = g(\bd c)$ for every $c \in C_2(X)$.
    It is clear that $\cbd f = 0$.
    Moreover, if $c \in C_2^R(X)$, with $R \ge R_0$, then $\lVert f(c) \rVert \le \lvert g \rvert_F^R \cdot \lvert \bd c \rvert_F^R \le \lvert g \rvert_F^R \cdot \lvert c \rvert_1$.
    This means that $f \in Z_\infsub^2(X;V)$.
    The assumption $\infcohom{X}{V}=0$ implies that $f$ has a primitive in $C_\infsub^1(X;V)$.
    This primitive is an extension of $g$ on $C_1(X)$ and is bounded on $C_1^R(X)$.
\end{proof}

\begin{lemma}\label{lemma:uniform_extensions}
  Let $X$ be a graph with a finite homological isoperimetric function, and suppose that $H^2_\infsub(X;\linfn) = 0$.
  Then for every large enough natural number $R$, there exists $K \in \mathbb{R}_{\ge 0}$ such that
  any bounded linear function $g\colon B_1^R(X) \to \mathbb{R}$ can be extended to a bounded linear function $G\colon C_1^R(X)\to \mathbb{R}$ with $\norm{G}_\infty^R \le K\cdot \norm{g}_F^R$.
\end{lemma}
\begin{proof}
  Take $R \ge R_0$, where $R_0$ is given by \Cref{lemma:extend_BC}.
  We argue by contradiction, assuming that there is a sequence $\{g_i\}_{i \in \mathbb{N}}$ of linear functions $g_i\colon B_1^R(X)\to\mathbb{R}$ with $\norm{g_i}_F^R \le 1$ such that every $G_i\colon C_1^R(X)\to\mathbb{R}$ extending $g_i$ has $\norm{G_i}_\infty^R \ge i$.

  Collecting the $g_i$ we obtain a linear function $g\colon B_1^R(X) \to \linfn$ with $\norm{g}_F^R \le 1$.
  By \Cref{lemma:extend_BC}, $g$ has a linear extension $G\colon C_1^R(X)\to\linfn$ with $\norm{G}_\infty^R < +\infty$.
  Such a $G$ provides an extension $G_i\colon C_1^R(X)\to\mathbb{R}$ for every $g_i$, with $\norm{G_i}_\infty \le \norm{G}_\infty$.
  This contradicts the choice of the $g_i$, for $i > \norm{G}_\infty$.
\end{proof}

We are now ready to prove Proposition \ref{prop:vanishing_hyp}.

\begin{proof}[Proof of Proposition \ref{prop:vanishing_hyp}]
    Fix a sufficiently large $R$ so that  \Cref{lemma:uniform_extensions} applies, and consider the inclusion $j\colon B_1^R(X) \to C_1^R(X)$.
    Now, \cite[Proposition 4.1]{Ger96}, applied to the map $j$, says that if there is a $K$ such that every bounded linear functional $g\colon B_1^R(X) \to \mathbb{R}$ can be extended to $G\colon C_1^R(X) \to \mathbb{R}$ with $\norm{G}_\infty^R \le K \norm{g}_F^R$, then $j$ is undistorted.
    In our situation, this is ensured by \Cref{lemma:uniform_extensions}.
    
    This means that $\norm{\cdot}_1 \ge C \norm{\cdot}_F^R$ on $B_1^R(X)$, for a certain $C > 0$ that may depend on $R$.
    That is, $X$ satisfies a homological linear isoperimetric inequality, and thus it is hyperbolic by \Cref{prop:linear_isop_hyp}.
\end{proof}

\section{Relative versions}\label{sec:relative}

In this section we study a relative version of $\ell^\infty$-cohomology. The main results are: Proposition \ref{prop:excision}, which is a form of excision; Corollary \ref{cor:cusp_hyp}, which is important for our cohomological characterisation of hyperbolically embedded subgroups; and Proposition \ref{prop:group_case}, relating relative $\ell^\infty$-cohomology and $\ell^\infty$-cohomology of groups.

In this section we consider pairs $(X,\mathcal{Y})$ in which $X$ is a metric space and $\mathcal{Y}$ is a collection of subspaces of $X$, possibly with multiplicities (in general, we allow a subset to appear multiple times in the collection).

The relevant notion of ``equivalence'' of pairs that we will be using is the following.

\begin{definition}
    Let $X$ and $X'$ be metric spaces. Let $\mathcal Y$ and $\mathcal Y'$ be collections of subspaces in $X$ and $X'$, respectively.
    A quasi-isometry of pairs $(f,f_\#)\colon (X,\mathcal{Y}) \to (X',\mathcal{Y}')$ is given by:
    \begin{itemize}
        \item A quasi-isometry $f\colon X \to X'$;
        \item A bijection $f_\#\colon \mathcal{Y} \to \mathcal{Y}'$ such that there is a constant $C \ge 0$ for which $d_\mathrm{Haus}(f(Y), f_\#(Y)) \le C$ for every $Y \in \mathcal{Y}$.
    \end{itemize}
\end{definition}

We now extend the definitions given in \Cref{sec:hyperbolic_graphs} to the relative setting.
Let $X$ be a graph and let $\mathcal Y$ be a collection of \emph{pairwise disjoint} subgraphs. For any normed vector space $V$ and every $i \ge 0$, we define
\[C^i(X,\mathcal Y;V)=\{f\in C^i(X;V) : f|_{(\mathcal{V}_Y)^{i+1}}=0\ \forall Y\in\mathcal Y\},\]
and its subspace
\[C^i_\infsub(X,\mathcal Y;V) = C^i(X,\mathcal Y;V) \cap C^i_\infsub(X;V).\]
By restricting the usual coboundary operator, we obtain a complex of vector spaces, and we denote its cohomology by $H^\bullet_\infsub(X,\mathcal Y;V)$.

\begin{remark}\label{rmk:only_vert}
    In the definition of $H^\bullet_\infsub(X,\mathcal Y;V)$, one can think of the elements $Y \in \mathcal{Y}$ as subsets of vertices of $X$: the edges they contain do not play any role.
\end{remark}

More generally, if $X$ is a graph and $\mathcal Y$ is a collection of subgraphs, we define $H_\infsub^\bullet(X,\mathcal Y;V)$ to be the corresponding cohomology of any pair $(X',\mathcal Y')$ which is quasi-isometric to $(X,\mathcal Y)$ and where the elements of $\mathcal Y'$ are pairwise disjoint subgraphs of $X'$.
This is well defined since such a pair $(X',\mathcal{Y}')$ always exists, and the cohomology is invariant under quasi-isometries of pairs:

\begin{lemma}\label{lemma:quasi_isometry}
    Let $X$ and $X'$ be connected graphs, and let $\mathcal{Y}$ and $\mathcal{Y}'$ be collections of pairwise disjoint subgraphs in $X$ and $X'$, respectively.
    Suppose that $(X,\mathcal Y)$ is quasi-isometric to $(X',\mathcal Y')$. Then, any quasi-isometry induces an isomorphism $H^k_\infsub(X,\mathcal Y;V) \cong H^k_\infsub(X',\mathcal Y';V)$.
\end{lemma}

\Cref{lemma:quasi_isometry} is a consequence of Lemmas \ref{lemma:relative_induced} and \ref{lemma:rel_close} below, which involve a more general class of maps, including quasi-isometries.

\begin{definition}\label{def:rel_coarse_uniform}
    Let $X$ and $X'$ be metric spaces, and let $\mathcal{Y}$ and $\mathcal{Y}'$ be collections of subspaces in $X$ and $X'$, respectively.
    A \emph{relative coarsely uniform map} $(f,f_\#)\colon (X,\mathcal{Y}) \to (X',\mathcal{Y}')$ is given by maps $f\colon X\to X'$ and $f_\#\colon \mathcal{Y}\to\mathcal{Y'}$ satisfying the following assumptions:
    \begin{enumerate}
        \item\label{it:sharp_sub} For every $Y \in \mathcal{Y}$, $f(Y) \subseteq f_\#(Y)$;
        \item\label{it:uniform} If $x_1,x_2 \in X$, then either there is some $Y \in\mathcal{Y}$ such that $x_1,x_2 \in Y$, or $d(f(x_1),f(x_2)) \le \rho_+(d(x_1,x_2))$.
    \end{enumerate}
    In condition \eqref{it:uniform}, $\rho_+\colon [0,\infty) \to [0,\infty)$ is a non-decreasing function, independent of $x_1$ and $x_2$.
\end{definition}

\begin{remark}
\label{rem:strict_containment}
    In condition \eqref{it:sharp_sub} it would make sense, and would probably be more natural, to only assume that $f(Y) \subseteq N_R(f_\#(Y))$ for some constant $R \ge 0$ independent of $Y$.
    In this way, the notion of relative coarsely uniform map would be stable under finite-distance perturbations of $f$.
    The stricter formulation of condition \eqref{it:sharp_sub} is needed in Lemmas \ref{lemma:relative_induced} and \ref{lemma:rel_close} below; after having proved these lemmas, one can safely relax the condition, keeping in mind that when performing a pull-back (see \Cref{lemma:relative_induced}), one has to perturb $f$ so that the stricter form of condition \eqref{it:sharp_sub} holds.
    The result will not depend on the perturbation, by \Cref{lemma:rel_close}.
\end{remark}

\begin{definition}\label{def:rel_close}
    Let $(X,\mathcal Y)$ and $(X',\mathcal{Y}')$ be pairs as in \Cref{def:rel_coarse_uniform}.
    Two relative coarsely uniform maps $(f,f_\#)$ and $(\hat f, \hat f_\#)$ from $(X,\mathcal Y)$ to $(X',\mathcal{Y}')$ are \emph{relatively close} if $f_\# = \hat f_\#$ and there exists a non-decreasing function $\rho\colon [0,\infty)\to[0,\infty)$ so that the following holds for every $x_1,x_2 \in X$:
    either there is some $Y \in \mathcal{Y}$ containing $x_1$ and $x_2$, or $d(f(x_1),\hat f(x_2)) \le \rho(d(x_1,x_2))$.
\end{definition}

In particular, if $f_\# =\hat f_\#$ and there is some $C \ge 0$ so that $d(f(x),\hat f(x)) \le C$ for every $x \in X$, i.e. $f$ and $\hat f$ are uniformly close in the usual sense, then $(f,f_\#)$ and $(\hat f, \hat f_\#)$ are relatively close, because in this situation $d(f(x_1),\hat f(x_2)) \le C + \rho_+(d(x_1,x_2))$ whenever $x_1$ and $x_2$ do not belong to the same $Y \in \mathcal{Y}$.

\begin{lemma}\label{lemma:relative_induced}
    Let $X$ and $X'$ be connected graphs, and let $\mathcal{Y}$ and $\mathcal{Y}'$ be collections of pairwise disjoint subgraphs in $X$ and $X'$, respectively.
    Then, any relative coarsely uniform map $(f,f_\#)\colon (X,\mathcal{Y}) \to (X',\mathcal{Y}')$ induces, by the usual pull-back of cochains, a map in cohomology $f^*\colon H^\bullet_\infsub(X',\mathcal Y';V) \to H^\bullet_\infsub(X,\mathcal Y;V)$.
\end{lemma}
\begin{proof}
    We need to prove that, given $\alpha \in C^k_\infsub(X',\mathcal Y';V)$, the pull-back $f^*\alpha$ belongs to $C^k_\infsub(X,\mathcal Y;V)$.
    
    Fix $D \ge 0$.
    Take $x_0, \dots, x_k \in X$ with $d(x_i,x_j) \le D$ for every $i,j \in \{0,\dots, k\}$.
    Consider first the case in which $x_0,\dots,x_k \in Y$ for some $Y \in \mathcal{Y}$.
    Then all $f(x_i)$ belong to $f_\#(Y)$, so $f^*\alpha(x_0,\dots,x_k) = \alpha(f(x_0),\dots,f(x_k)) = 0$.
    This already proves that $f^*\alpha \in C^k(X,\mathcal Y;V)$.
    
    Now, consider the case in which there is no such $Y$.
    If two indices $i$ and $j$ are such that $x_i$ and $x_j$ do not belong to the same $Y \in \mathcal{Y}$, then $d(f(x_i),f(x_j)) \le \rho_+(d(x_i),d(x_j)) \le \rho_+(D)$.
    If, instead, $x_i,x_j \in Y$ for some $Y$, there must be a third index $x_l$ that does not belong to $Y$, and by the triangle inequality, we have $d(f(x_i),f(x_j)) \le d(f(x_i),f(x_l)) + d(f(x_l),f(x_j)) \le 2\rho_+(D)$.
    
    In any case, we have $d(f(x_i),f(x_j)) \le 2\rho_+(D)$ for every $i$ and $j$.
    It follows that $f^*\alpha$ belongs to $C_\infsub^k(X;V)$. Thus, it belongs to $C^k_\infsub(X,\mathcal Y;V)$.
\end{proof}

The following lemma concludes the proof of Lemma \ref{lemma:quasi_isometry}, by considering a quasi-inverse of the quasi-isometry.

\begin{lemma}\label{lemma:rel_close}
    Let $(f,f_\#)$ and $(\hat f, \hat f_\#)$ be relatively coarsely uniform maps from $(X,\mathcal{Y})$ to $(X',\mathcal{Y}')$, where, as in \Cref{lemma:relative_induced}, the subspaces in $\mathcal{Y}$ and $\mathcal{Y}'$ are all pairwise disjoint.
    If the two maps are relatively close, then they induce the same pull-back map $H^\bullet_\infsub(X',\mathcal Y';V) \to H^\bullet_\infsub(X,\mathcal Y;V)$.
\end{lemma}
\begin{proof}
    We consider the usual homotopy maps $h\colon C^{k+1}(X';V) \to C^{k}(X;V)$ defined by the formula
    \[(h\alpha)(x_0,\dots,x_{k}) = \sum_{i=0}^k (-1)^i \alpha(f(x_0),\dots,f(x_i),\hat f(x_i), \dots,\hat f(x_{k})),\]
    and prove that if $\alpha \in C^{k+1}_\infsub(X',\mathcal Y';V)$, then $h\alpha \in C^{k}_\infsub(X,\mathcal Y;V)$. 
    
    If all the $x_i$ belong to $Y \in \mathcal{Y}$, then every $f(x_i)$ and every $\hat f(x_i)$ belong to $f_\#(Y) = \hat f_\#(Y)$.
    It follows that $h\alpha \in C^k(X,\mathcal Y;V)$.
    
    Suppose now that the pairwise distances between the $x_i$ are not bigger than some constant $D$, and that there is no $Y \in\mathcal Y$ containing all of them.
    Since the subsets in $\mathcal{Y}$ are pairwise disjoint, there are two indices $i^*,j^* \in \{0,\dots,k\}$ such that there is no $Y \in \mathcal{Y}$ containing both $x_{i^*}$ and $x_{j^*}$.
    Consider now any two indices $i$ and $j$.
    As in the proof of \Cref{lemma:relative_induced}, we have $d(f(x_i),f(x_j)) \le 2\rho_+(D)$ and $d(\hat f(x_i),\hat f(x_j)) \le 2\hat\rho_+(D)$.
    Moreover, since $f$ and $\hat f$ are relatively close, we have
    \begin{align*}
        d(f(x_i),\hat f(x_j)) &\le  d(f(x_i),f(x_{i^*}))
        + d(f(x_{i^*}), \hat f(x_{j^*}))
        + d(\hat f(x_{j^*}), \hat f(x_j))\\
        &\le 2\rho_+(D) + \rho(D) + 2\hat\rho_+(D).
    \end{align*}
    Therefore, $h\alpha \in C_\infsub^k(X,\mathcal Y;V)$ and the proof is complete.
\end{proof}

\begin{example}
    Let $X$ be the standard Cayley graph of the group $\mathbb{Z}$ of integers.
    Take $\mathcal{Y} = \{X\}$; then, it follows easily from the definitions that $H_\infsub^k(X,\mathcal{Y};\mathbb{R}) = 0$ in every degree $k \ge 0$.
    On the other hand, if $\mathcal{Y}' = \{X, X\}$, so that $X$ appears with multiplicity $2$, it can be seen that $H_\infsub^2(X,\mathcal{Y'};\mathbb{R}) \neq 0$; this computation can be done by using the pair $(Z,\{Y_1,Y_2\})$ displayed in \Cref{fig:example_Z}, which is quasi-isometric to $(X,\mathcal{Y}')$.
    \begin{figure}
        \centering
        \includegraphics[scale=0.8]{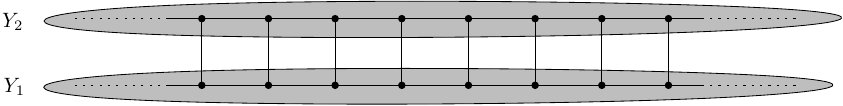}
        \caption{A graph $Z$ with two subgraphs $Y_1,Y_2$.}
        \label{fig:example_Z}
    \end{figure}
    Consider the map $(f,f_\#)\colon (Z,\{Y_1,Y_2\}) \to (Z,\{Y_1,Y_2\})$ where $f_\#(Y_1) = f_\#(Y_2) = Y_1$ and $f$ is the obvious projection onto $Y_1$.
    Notice that $f\colon Z \to Z$ is uniformly close to the identity.
    However, it is not relatively close, in the sense of \Cref{def:rel_close}, to the identity of the pair $(Z,\{Y_1,Y_2\})$, and the induced map in $H_\infsub^2$ is the zero map; in fact, $(f,f_\#)$ factors through $(X,\mathcal{Y})$.

    This example shows the importance of taking models with disjoint subgraphs, that multiplicities of subgraphs are relevant, and that it is important, when performing pull-backs of cochains, to use maps $f$ satisfying assumption \eqref{it:sharp_sub} in \Cref{def:rel_coarse_uniform}, asking for sharp containment.
\end{example}

\subsection{Excision}

The following proposition is a form of excision for relative $\ell^\infty$-cohomology, and will be used to study cusped spaces.

\begin{prop}[Excision]
\label{prop:excision}
    Let $X$ and $X'$ be connected graphs, and let $\mathcal{Y}$ and $\mathcal{Y}'$ be collections of pairwise disjoint subgraphs in $X$ and $X'$, respectively.
    Let $(f,f_\#)\colon (X,\mathcal{Y}) \to (X',\mathcal{Y}')$ be a relatively coarsely uniform map, satisfying the following additional assumptions:
    \begin{itemize}
        \item $f_\#$ is a bijection;
        \item For every $Y \in \mathcal{Y}$ and every $x \in X$, we have $x \in Y$ if and only if $f(x) \in f_\#(Y)$;
        \item There is a non-decreasing function $\rho_-\colon [0,\infty) \to [0,\infty)$, with $\rho_-(r) \to \infty$ as $r \to \infty$, such that, whenever $x_1,x_2 \in X$ do not both belong to some $Y \in \mathcal{Y}$, then $d(f(x_1),f(x_2)) \ge \rho_-(d(x_1,x_2))$;
        \item There is a non-decreasing function $\rho\colon [0,\infty)\to [0,\infty)$ such that, for every $Y' \in \mathcal{Y}'$ and every $y' \in Y'$ and $x' \in X' \setminus Y'$, there is some $y \in X$ such that $d(y',f(y)) \le \rho(d(y',x'))$.
    \end{itemize}
    Moreover, assume that $X' = f(X) \cup \bigcup_{Y' \in \mathcal{Y}'} Y'$.
    Then, in any degree $k\ge 0$, the pull-back map $f^*\colon H_\infsub^k(X',\mathcal{Y'};V) \to H_\infsub^k(X,\mathcal{Y};V)$ is an isomorphism.
\end{prop}
\begin{proof}
    The assumption about $\rho_-$ implies the existence of a non-decreasing map $\rho^*_-\colon [0,\infty) \to [0,\infty)$ such that $d(x_1,x_2) \le \rho^*_-(d(f(x_1),f(x_2)))$ whenever $x_1,x_2$ do not belong to the same $Y \in \mathcal{Y}$.
    Recall also that there is a map $\rho_+\colon [0,\infty) \to [0,\infty)$ such that, under the same condition, $d(f(x_1),f(x_2)) \le \rho_+(d(x_1,x_2))$.

    We proceed by defining a map of pairs $(\pi,\pi_\#)\colon (X',\mathcal{Y}') \to (X,\mathcal{Y})$, which will be a ``relative coarse inverse'' of $(f,f_\#)$.
    Define $\pi_\#\colon \mathcal{Y}'\to\mathcal{Y}$ as the inverse of $f_\#$, and define $\pi\colon X' \to X$ as follows:
    \begin{itemize}
        \item If $y' \in Y'$ for some $Y' \in \mathcal{Y}'$, then $\pi(y') = y$, where $y \in \pi_\#(Y')$ minimizes the distance $d(y',f(y))$;
        \item If $x'$ does not belong to any such $Y'$, then $\pi(x') = x$ where $x \in X$ is some point with $f(x) = x'$.
    \end{itemize}
    
    We now check that $(\pi,\pi_\#)$ is relatively coarsely uniform.
    Take $x'_1,x'_2 \in X'$ not belonging to a common $Y' \in \mathcal{Y}'$, with $d(x'_1,x'_2) \le D$.
    We treat here the case where $x'_1 \in Y'_1$ and $x'_2 \in Y'_2$ for some distinct $Y'_1,Y'_2 \in \mathcal{Y}$, the other cases being similar and easier.
    Set $y_1 = \pi(x'_1)$ and $y_2 = \pi(x'_2)$.
    Then,
    \begin{align*}
    d(y_1,y_2) &\le \rho^*_-(d(f(y_1),f(y_2)))\\
    &\le \rho^*_-(d(f(y_1),x'_1) + d(x'_1,x'_2) + d(x'_2,f(y_2)))\\
    &\le \rho^*_-(\rho(D) + D + \rho(D)).
    \end{align*}

    Next, we check that the composition $(f\circ\pi,f_\#\circ\pi_\#)\colon (X',\mathcal{Y'}) \to (X',\mathcal{Y'})$ is relatively close to the identity.
    By construction, $f_\#\circ\pi_\#\colon \mathcal{Y}' \to \mathcal{Y}'$ is the identity.
    Take $x'_1,x'_2 \in X'$ as in the previous paragraph.
    Consider the case in which $x'_2 \in Y' \in \mathcal{Y}'$, leaving the other case to the reader.
    Set $y_2 = \pi(x'_2)$.
    Then,
    \[
    d(x'_1, f\circ\pi(x'_2)) =d(x'_1, f(y_2)) \le \rho(d(x'_1,x'_2)) \le \rho(D).
    \]
    If instead $x'_2$ does not belong to any $Y' \in \mathcal{Y}$, then $f\circ\pi(x'_2) = x'_2$, and $d(x'_1, f\circ\pi(x'_2)) \le D$.

    Then, we check the composition $(\pi\circ f,\pi_\#\circ f_\#)\colon  (X,\mathcal{Y}) \to (X,\mathcal{Y})$.
    Again, $\pi_\#\circ f_\#$ is the identity by construction.
    Take $x_1,x_2 \in X$ with $d(x_1,x_2) \le D$, not belonging to the same $Y \in \mathcal{Y}$.
    Then, $\hat x_2 = \pi\circ f(x_2) \in X$ is such that $f(x_2) = f(\hat x_2)$, and we have
    \[
    d(x_1, \hat x_2) \le \rho^*_-(d(f(x_1),f(\hat x_2))) = \rho^*_-(d(f(x_1),f(x_2))) \le \rho^*_-(\rho_+(D)).
    \]

    With both compositions being relatively close to the respective identity maps, we conclude by applying \Cref{lemma:rel_close}.
\end{proof}

We will use the following proposition to obtain the vanishing of relative $\ell^\infty$-cohomology, in the case where $X$ is hyperbolic.

\begin{prop}
\label{prop:rel_cohom_surj}
    Let $X$ be a connected graph and let $\mathcal Y$ be a collection of uniformly hyperbolic connected subgraphs. Then there exists a surjective map $\infcohom{X,\mathcal Y}{V}\to \infcohom{X}{V}$.
\end{prop}

\begin{proof}
    Up to replacing $(X,\mathcal Y)$ with a quasi-isometric pair, we can assume that all $Y \in \mathcal Y$ are disjoint. 
    
    Consider any cocycle $f$ on $X$. We need to find a cocycle cohomologous to $f$ which vanishes on all $Y\in\mathcal Y$. By Proposition \ref{prop:hyp_vanishing}, each restriction of $f$ to a $Y\in\mathcal Y$ is the coboundary of some $\phi_Y$, and the norm of $\phi_Y$ is bounded in terms of the norm of $f$. We can extend each $\phi_Y$ to $X$, by setting it to $0$ on simplices not contained in $Y$; we still denote the extension by $\phi_Y$. Then $f-\sum \delta\phi_Y$ is cohomologous to $f$ and it vanishes on all $Y\in\mathcal Y$, as required.
\end{proof}

We now recall the definition of the cusped space $\cusp{X,\mathcal Y}$, which is obtained by gluing combinatorial horoballs onto the various $Y\in\mathcal Y$.

\begin{definition}
    Let $Y$ be a graph.
    The combinatorial horoball with basis $Y$ is the graph with vertex set $\mathcal{V}_Y \times \mathbb{N}$ and edges of the following types:
    \begin{itemize}
        \item For every $y \in \mathcal{V}_Y$ and $n \in \mathbb{N}$, the vertices $(y,n)$ and $(y,n+1)$ are adjacent;
        \item For every $n \in \mathbb{N}$, and every pair of distinct vertices $y_1,y_2 \in \mathcal{V}_Y$, the vertices $(y_1,n)$ and $(y_2,n)$ are adjacent if and only if $d_Y(y_1,y_2) \le 2^n$.
    \end{itemize}

\end{definition}

\begin{definition}
    Let $X$ be a graph and $\mathcal{Y}$ be a collection of subgraphs of $X$.
    The cusped space $\cusp{X,\mathcal Y}$ associated to the pair $(X,\mathcal{Y})$ is the graph obtained from the disjoint union $X \sqcup \bigsqcup_{Y\in \mathcal{Y}} \mathcal{H}_Y$ by adding, for every $Y \in \mathcal{Y}$ and $y \in Y$, an edge joining $y \in X$  to $(y,0) \in \mathcal{H}_Y$.
    
    We denote by $\mathcal{H}_\mathcal{Y}$ the collection $\{\mathcal{H}_Y: Y \in \mathcal{Y}\}$, which is a family of disjoint subgraphs of $\cusp{X,\mathcal{Y}}$.
\end{definition}

Hyperbolicity of a cusped space is equivalent to (relative) metric hyperbolicity, and this is why the following is of interest to us.

\begin{prop}
\label{prop:cusp_rel_cohomo_0}
    Let $X$ be a connected graph and let $\mathcal Y$ be a collection of connected subgraphs. Suppose that $\cusp{X,\mathcal Y}$ is hyperbolic. Then $\infcohom{\cusp{X,\mathcal Y},\mathcal H_\mathcal Y}{V}=0$ for every $1$-injective Banach space $V$.
\end{prop}

In order to prove Proposition \ref{prop:cusp_rel_cohomo_0} we will use the following result, which we prove in the next section, as it requires a separate set of tools. The result is about extending 1-cocycles defined on subgraphs to the whole graph, and it uses the axiomatic setup of \cite{BBF}, and for the proof we will use its refinement from \cite{BBFS}.

\begin{prop}
\label{prop:extend}
    Let $X$ be a graph and let $\mathcal Y$ be a collection of disjoint connected full subgraphs. Suppose that we have assigned to each $Y\in\mathcal Y$ a map $\pi_Y\colon X\to 2^Y$ satisfying the following properties for some constant $B$, where we denote $d_Y(\cdot,\cdot)=diam(\pi_Y(\cdot)\cup\pi_Y(\cdot))$.
    \begin{enumerate}
    \item (Bounded projection) If $W\in\mathcal Y$ and either $Y\in\mathcal Y$ is distinct from $W$ or $Y\in X$, then $\pi_W(Y)$ has diameter at most $B$.
    \item (Coarse Lipschitz)  For all $x,y\in X$, we have $d_Y(x,y)\leq B d(x,y)+B$.
        \item (Behrstock inequality) If $W,Y\in\mathcal Y$ are distinct and $x\in X$, then
        $$\min \{d_W(Y,x)),d_Y(W,x))\}\leq B.$$
        \item (Large projections) If $W,Y\in\mathcal Y$ are distinct, then 
        $$|\{Z: d_Z(W,Y)\geq B\}|<+\infty.$$
    \end{enumerate}
    Consider any family of $1$-cocycles $\{\phi_Y\}$ on the $Y\in \mathcal Y$ such that the restrictions $\phi_Y|_{C^1_1(Y,V)}$ have uniformly bounded norm. Then there exists a $1$-cocycle $\phi$ on $X$ which extends all $\phi_Y$.
\end{prop}

\begin{proof}[Proof of Proposition \ref{prop:cusp_rel_cohomo_0}]

    Consider any 2-cocycle $f$ on $\cusp{X,\mathcal Y}$ which vanishes on every horoball $\mathcal H_Y \in \mathcal H_\mathcal{Y}$.
    In view of Proposition \ref{prop:hyp_vanishing}, we have that $f=\delta \phi$ for some 1-cochain $\phi$ on $\cusp{X,\mathcal Y}$.
    Since $f$ vanishes on every horoball $\mathcal{H}_Y \in\mathcal H_\mathcal{Y}$, the restriction $\phi_Y$ of $\phi$ to such an $\mathcal{H}_Y$ is a cocycle.
    We now wish to apply Proposition \ref{prop:extend} to find a 1-cocycle $\psi$ on $\cusp{X,\mathcal Y}$ which restricts to $\phi_Y$ on each $\mathcal{Y}\in\mathcal H_\mathcal Y$. Given such a $\psi$, we have that $\delta(\phi-\psi)=f$, and $\phi-\psi$ vanishes on every horoball, proving the assertion.

    Hence, it remains to prove that Proposition \ref{prop:extend} can be applied in our situation, on the graph $\cusp{X,\mathcal{Y}}$ and the family of subgraphs $\mathcal H_\mathcal{Y}$.
    We consider the projection maps $\pi_{\mathcal{H}_Y}\colon \cusp{X,\mathcal{Y}} \to 2^{\mathcal{H}_Y}$  that send a vertex $v \in \cusp{X,\mathcal{Y}}$ to the set of vertices in $\mathcal{H}_Y$ with minimum distance from $v$.
    \begin{itemize}
        \item \emph{$\pi_{\mathcal{H}_Y}(v)$ has finite diameter, with a uniform bound independent of $v$ and $Y$.}
        We can assume that $v \notin \mathcal{H}_Y$.
        Points in $\pi_{\mathcal H_Y}(v)$ are of the form $(y,0) \in \mathcal{H}_Y$.
        Take two such points $(y_1,0)$, $(y_2,0)$ minimising the distance from $v$.
        Geodesics from $v$ to these two points can be prolonged deeper in the horoball, reaching vertices $(y_1,m)$ and $(y_2,m)$ which, if $m \in \mathbb{N}$ is fixed sufficiently big, are adjacent.
        Prolong further the two geodesics, arriving to a common endpoint coinciding with the midpoint of the edge joining $(y_1,m)$ to $(y_2,m)$.
        We have formed a bigon, whose width is bounded from above because of the hyperbolicity of $\cusp{X,\mathcal Y}$; in turn, this gives a bound for the distance between $(y_1,0)$ and $(y_2,0)$.
        \item \emph{The horoballs $\mathcal{H}_Y$ are uniformly quasi-convex.}
        Take $(y_1,n_1), (y_2,n_2) \in \mathcal{H}_Y$.
        Then, take a big enough $m \in \mathbb{N}$ so that, going deeper in the horoball, $(y_1,m)$ and $(y_2,m)$ are adjacent (the case $y_1 = y_2$ is easier).
        Notice that the geodesics in $\cusp{X,Y}$ joining $(y_1,n_1)$ to $(y_1,m)$ and $(y_2,n_2)$ to $(y_1,m)$ are contained in $\mathcal{H}_Y$.
        Therefore, any geodesic from $(y_1,n_1)$ to $(y_2,n_2)$ must be contained in a fixed neighbourhood of $\mathcal{H}_Y$ that depends only on the hyperbolicity constant of $\cusp{X,\mathcal Y}$.
        \item \emph{Given distinct $Y,Z \in \mathcal{Y}$, the projection $\pi_{\mathcal{H}_Y}(\mathcal{H}_Z)$ has uniformly bounded diameter.}
        Suppose that, for arbitrarily large $D$, there are points $v_1,v_2 \in \mathcal{H}_Z$ and $w_1,w_2 \in \mathcal{H}_Y$ with $w_i \in \pi_{\mathcal{H}_Y}(v_i)$ and $d(w_1,w_2) \ge D$.
        Prolong a geodesic from $v_1$ to $w_1 = (y_1,0)$ deeper in the horoball, reaching a certain $(y_1,m)$, and do the same with $v_2$ and $w_2 = (y_2,0)$, reaching $(y_2,m)$, so that the new endpoints are adjacent.
        Consider a geodesic triangle with vertices $v_1,v_2$ and the midpoint of the edge from $(y_1,m)$ and $(y_2,m)$.
        Since we know that $w_1$ and $w_2$ are far apart (if $D$ is big enough), the hyperbolicity of $\cusp{X,\mathcal Y}$ implies that there is a uniform $R$ (depending only on the hyperbolicity constant) such that both $w_1$ and $w_2$ are at distance $\le R$ from the geodesic $[v_1,v_2]$.
        Since horoballs are quasi-convex, by possibly changing $R$ with a bigger constant, we conclude that $w_1$ and $w_2$ are at distance at most $R$ from $\mathcal{H}_Z$.
        
        Summarizing, there is a constant $R$ such that, for arbitrarily large values of $D$, there are pairs of vertices $w_1,w_2 \in \mathcal{H}_Y$ and $u_1,u_2 \in \mathcal{H}_Z$ with $d(w_1,w_2) \ge D$, $d(w_1,u_1) \le R$ and $d(w_2,u_2) \le R$.
        Let $\delta$ be the hyperbolicity constant of $\cusp{X;\mathcal Y}$ (so that triangles are $\delta$-thin and $\delta$-slim, as in \cite[Remark 2.10]{GM2008}).
        Up to replacing $R$ with $R + \delta$, we can assume that the points $w_1$ and $w_2$ are at depth bigger than $\delta$ in $\mathcal{H}_Y$.
        This implies (see \cite[Lemma 3.26]{GM2008}) that geodesics from $w_1$ to $w_2$ stay in $\mathcal{H}_Y$, and reach depths of order $\log(D)$ in the horoball.
        Hence, such a geodesic is not uniformly close to a geodesic from $u_1$ to $u_2$ (which must stay $\delta$-close to $\mathcal{H}_Z$), contradicting the hyperbolicity of $\cusp{X,\mathcal{Y}}$.
    \end{itemize}
    The conclusion follows because any family of quasi-convex subspaces of a hyperbolic space, with uniformly bounded projections onto each other, satisfies the hypotheses of Proposition \ref{prop:extend}, see, e.g. the proof of \cite[Lemma 4.47]{DGO17}.
\end{proof}

We note that putting together Proposition \ref{prop:cusp_rel_cohomo_0}, Proposition \ref{prop:rel_cohom_surj}, and Theorem \ref{thm:hyp_main} we get:

\begin{corollary}
\label{cor:cusp_hyp}
    Let $X$ be a connected graph and let $\mathcal Y$ be a collection of connected subgraphs. Suppose that $\cusp{X,\mathcal Y}$ has finite isoperimetric function. Then $\cusp{X,\mathcal Y}$ is hyperbolic if and only if $\infcohom{\cusp{X,\mathcal Y},\mathcal H_\mathcal Y}{V}=0$ for every $1$-injective Banach space $V$.
\end{corollary}

\subsection{The group case}
Let $G$ be a group and let $\mathcal{H} = \{H_i\}_{i\in I}$ be a parametrised family of subgroups of $G$ (repetitions are allowed).
For any $G$-module $M$, Bieri and Eckmann \cite{BE1978} defined the relative cohomology $H^k(G,\mathcal{H};M)$.
By taking $M = \ell^\infty(G,V)$, where $V$ is a normed vector space, one obtains the ``bounded-valued'' (or ``$\ell^\infty$'') relative cohomology $H_\infsub^k(G,\mathcal{H};V) = H^k(G,\mathcal{H};\ell^\infty(G,V))$ which has been considered in \cite{Mil2021} (generalising the definition given previously in \cite{GH2009}).
Here, $\ell^\infty(G,V)$ is the vector space of functions $G\to V$ with bounded image, endowed with the usual action of $G$ by multiplication (on the left) on the argument.

We now relate relative $\ell^\infty$-cohomology and the relative cohomology discussed above.

\begin{prop}\label{prop:group_case}
    Let $G$ be a finitely generated group, and let $\mathcal{H} = \{H_i\}_{i\in I}$ be a finite family of subgroups of $G$.
    Let $X$ be the Cayley graph of $G$ with respect to a finite generating set, and let $\mathcal{Y}$ be the family of subgraphs whose members correspond to the cosets $gH_i$, for every $i \in I$ and $gH_i \in G/H_i$.
    Then, $H_\infsub^k(G,\mathcal{H};V) \cong H_\infsub^k(X,\mathcal{Y};V)$ for any $k \ge 0$ and any normed vector space $V$.
\end{prop}
\begin{proof}
    For any $\Gamma$-module $M$, the relative cohomology $H^k(G,\mathcal{H};M)$ can be realised as the cohomology of the following standard cochain complex which has been considered, e.g. in \cite{MY,Fra2018}.
    For any $k \ge 0$, define $C^k(G,\mathcal{H};M)$ as the space of functions $\alpha\colon (G\times I)^{k+1} \to M$ satisfying the following properties:
    \begin{itemize}
        \item $\alpha$ is $G$-equivariant ($G$ acts on $G\times I$ by multiplication on the left on the first factor, and on $(G\times I)^{k+1}$ diagonally);
        \item $\alpha((g_0,i), \dots, (g_k,i)) = 0$ (note that all second coordinates are equal) when $g_j H_{i} = g_0 H_{i}$ for every $j$.
    \end{itemize}
    Then, the relative cohomology is obtained by taking the cohomology of the complex $C^\bullet(G,\mathcal{H};M)$, with the usual coboundary maps performing alternating sums, erasing one argument at a time.

    In our case, $M = \ell^\infty(G,V)$, and $C^k(G,\mathcal{H};M) = C^k(G,\mathcal{H};\ell^\infty(G,V))$ can be described equivalently as the space of functions $\alpha\colon (G\times I)^{k+1} \to V$ where the same vanishing condition as before is imposed, but we waive the $G$-equivariance property by replacing it with the following:
    \begin{itemize}
        \item $\alpha$ is bounded on every orbit of the action of $G$ on $(G\times I)^{k+1}$.
    \end{itemize}
    We denote this space of functions by $C_\infsub^k(G,\mathcal{H};V)$.
    The isomorphism between $C_\infsub^k(G,\mathcal{H};V)$ and $C^k(G,\mathcal{H};\ell^\infty(G,V))$ is obtained by evaluating $\ell^\infty(G,V)$-valued cochains on the identity of $G$; see, e.g. \cite{Mil2021} for details.
    
    Consider a graph $X'$ with vertex set $G\times I$ and edges of two kinds:
    \begin{itemize}
        \item $(g,i_1)$ is adjacent to $(g,i_2)$ for every $g \in G$ and $i_1,i_2 \in I$;
        \item $(g,i)$ is adjacent to $(gs,i)$ for every $g \in G$, $i \in I$ and $s$ in the fixed finite generating set.
    \end{itemize}
    By construction, each layer $G\times\{i\}$ is isomorphic to the Cayley graph $X$.
    For every $i \in I$ and every $gH_i \in G/H_i$, consider the subgraph of $X'$ induced by the subset of vertices of the form $(gh,i)$ where $h$ varies in $H_i$, and collect these subgraphs into a family $\mathcal{Y}'$.
    
   The members of $\mathcal{Y}'$ are pairwise disjoint, and the pair $(X,\mathcal{Y})$ is quasi-isometric to $(X',\mathcal{Y}')$: a quasi-isometry is obtained by identifying $X$ with a layer of $X'$ with a fixed $I$-coordinate.
 By definition of $H_\infsub^k$ of a pair, we have $H_\infsub^k(X,\mathcal{Y};V) \cong H_\infsub^k(X',\mathcal{Y}';V)$.

    To conclude, we notice that the complexes $C_\infsub^\bullet(G,\mathcal{H};V)$ and $C_\infsub^\bullet(X',\mathcal{Y}';V)$ are isomorphic.
    In any degree $k \ge 0$, they consist of $V$-valued functions on $(G\times I)^{k+1}$, with the same vanishing assumption, that are bounded on $G$-orbits or, respectively, bounded on sets of tuples with fixed diameter.
    However, tuples in the same $G$-orbit always have a fixed diameter. On the other hand, a set of tuples with fixed diameter is contained in the union of a finite number of $G$-orbits, because $I$ is finite and $G$ is finitely generated, and we have used a finite generating set in the construction of the graph.
\end{proof}

\section{Extending 1-cocycles}\label{sec:extending}

The goal of this section is to prove Proposition \ref{prop:extend}.

We will extend the 1-cocycles by extending (unbounded, but Lipschitz) primitives provided by the following lemma.

\begin{lemma}
\label{lem:lip_primitive}
    Let $\phi\in C^1(X;V)$ be a 1-cocycle on some connected graph $X$. Then $\phi=\delta f$ for some function $f\colon X\to V$, with Lipschitz constant being the norm of $\phi|_{C_1^1(X,V)}$.
\end{lemma}

\begin{proof}
Fix a base vertex $x_0\in X$ and define $f(x):=\phi(x_0,x)$. Then for all vertices $x,y\in X$, we have $\delta f(x,y)=f(y)-f(x)=\phi(x_0,y)-\phi(x_0,x)=\phi(x,y)$, where the last equality is because $\phi(\partial (x_0,x,y))=0$.

Also, the fact that the function $f$ is Lipschitz with the required constant follows immediately from the fact that for all adjacent vertices $x,y$ we have $|f(y)-f(x)|=|\phi(x,y)|\leq |\phi|^1_\infty$.
\end{proof}

Even though we do not need this, we note the following consequence:

\begin{corollary}
    $H^1_{(\infty)}(X;V)$ is isomorphic to the space of Lipschitz functions on $X$, modulo bounded functions.
\end{corollary}

We are now ready to prove the proposition.

\begin{proof}[Proof of Proposition \ref{prop:extend}]
We can assume that the subgraphs $\mathcal Y$ cover the whole vertex set of the graph by adding single vertices as subgraphs of $\mathcal Y$, with corresponding map $\pi_Y$ being the only possible one. All properties are easily seen to be preserved.

Also, by \cite[Section 4]{BBFS}, we can assume, up to increasing $B$, that a property stronger than (3) holds, namely

\par\medskip

(3') If $W,Y,Z$ are distinct and $d_W(Y,Z)\geq B$, then $\pi_Y(W)=\pi_Y(Z)$.

\par\medskip

Now, in view of Lemma \ref{lem:lip_primitive} we have Lipschitz functions $f_Y$ on each $Y\in\mathcal Y$ with uniformly bounded Lipschitz norm (with respect to the intrinsic metric of $Y$, whence the ambient metric of $X$), and we would like to construct a Lipschitz function $f$ on $X$, such that for all $Y\in\mathcal Y$, we have that $f|_Y-f_Y$ is constant. In fact, we have to define $f$ on each $Y$, since together they cover $X$, and since the various $Y$ are disjoint, we can define $f$ on each of them separately.

For $W,Y\in\mathcal Y$ denote
\[\prel{W}{Y}=\{Z: d_Z(W,Y)> 10B\}\cup \{W,Y\}\]
and $\rho(W,Y)=|\prel{W}{Y}|$; this cardinality is finite by (4). There is a total order $\preceq$ on $\prel{W}{Y}$ with maximal element $Y$ given in \cite[Lemma 2.2]{BBFS}, which we will use below. There are several equivalent characterisations of the order, including that $U\prec V$ if and only if $\pi_U(V)=\pi_U(Y)$.

Fix any $W_0\in \mathcal Y$ and define $f=f_{W_0}$ on $W_0$. Inductively, suppose that we have defined $f$ on any $Y$ such that $\rho(W,W_0)\leq n$ (note that $W=W_0$ is the only element of $\mathcal Y$ with $\rho(W,W_0)=1$). Let $Y\in\mathcal Y$ be such that $\rho(Y,W_0)=n+1$, and let $Y'=p(Y)$ be the penultimate element in the $\preceq$ order on $Rel(W_0,Y)$. By the following claim we have $\rho(p(Y),W_0)=n$, so that $f$ has been defined on $p(Y)$.

\par\medskip

{\bf Claim 1.} $\prel{W_0}{p(Y)}$ consists of all elements $\chi\in\prel{W_0}{Y}$ with $\chi\preceq p(Y)$.

\begin{proof}
    This follows from \cite[Proposition 2.3, Corollary 2.5]{BBFS}.
\end{proof}

Consider any $b(Y)\in\pi_Y(Y')$ and $s(Y)\in \pi_{Y'}(Y)$ (``$b$'' for basepoint and ``$s$'' for shadow). Define $f$ on $Y$ as
$$f_Y-f_Y(b(Y))+f(s(Y)).$$

We are left with checking that $f$ is Lipschitz. Let $y,z$ be adjacent vertices of $X$, with $y\in Y$ and $z\in Z$ for some $Y,Z\in \mathcal Y$, and let us show that there is a bound on $|f(y)-f(z)|$. We can assume that $Y$ and $Z$ are distinct, for otherwise we can use that $f_Y=f_Z$ is Lipschitz.

We need two claims on the structure of $\prel{\cdot}{\cdot}$-sets and their order. Both claims also hold switching the roles of $Y$ and $Z$. Note that the only element $W$ of $\mathcal Y$ for which $p(W)$ is not defined is $W_0$.

\par\medskip

{\bf Claim 2.} If $Y,p(Y)\neq W_0$, then $p(p(Y))$ is the element of $\prel{W_0}{Y}$ preceding $p(Y)$ with respect to $\preceq$.

\begin{proof}
     This follows immediately from Claim 1.
\end{proof}

\par\medskip

{\bf Claim 3.} $\prel{W_0}{Y}-\prel{W_0}{Z}$ is either empty, $\{Y\}$ or $\{Y,p(Y)\}$.

\begin{proof}
If $Y\in \prel{W_0}{Z}$, then $\prel{W_0}{Y}\subseteq \prel{W_0}{Z}$ by \cite[Corollary 2.5]{BBFS}. If $Y\notin \prel{W_0}{Z}$ and $p(Y)\in \prel{W_0}{Z}$, similarly any $W\in \prel{W_0}{Y}$ with $W\preceq p(Y)$ is in $\prel{W_0}{Z}$.

Hence, suppose $Y,p(Y)\notin \prel{W_0}{Z}$, and let us show that all elements of $\prel{W_0}{Y}$ which are $\prec p(Y)$, are in fact exactly $\prel{W_0}{Z}$.

We have $d_{p(Y)}(W_0,Z)\geq 5B$, since $d_{p(Y)}(W_0,Y)\geq 10B$ (by definition of $\prel{\cdot}{\cdot}$) and $d_{p(Y)}(Y,Z)\leq 4B$ by the coarse Lipschitz property of $\pi_{p(Y)}$ and bounded projections. So $p(Y)$ lies in the analogue of $\prel{W_0}{Z}$ defined replacing $10B$ with $5B$ and once again we can use \cite[Proposition 2.3, Corollary 2.5]{BBFS} (which apply with the lower constant $5B$ as well).
\end{proof}

Consider now $Y_0=Y$, $Y_1=p(Y)$, $Y_2=p(p(Y))$ (where these are defined), and similarly for $Z$. 
Up to switching the roles of $Y$ and $Z$, we can assume that the minimal $i$ such that $Y_i\in \prel{W_0}{Y}\cap \prel{W_0}{Z}$ is no larger than the corresponding index for $Z$.

Then, because of the claims, one of the following must hold:

\begin{itemize}
\item $Y\in \prel{W_0}{Z}$ and $p(Z)=Y$,
\item $Y\in \prel{W_0}{Z}$ and $p(p(Z))=Y$,
\item $Y\notin \prel{W_0}{Z}$, $p(Y)\in \prel{W_0}{Z}$, and $p(Z)=p(Y)$,
\item $Y, p(Y)\notin \prel{W_0}{Z}$ and $p(p(Z))=p(Y)$,
\item $Y, p(Y)\notin \prel{W_0}{Z}$ and $p(p(Z))=p(p(Y))$.
\end{itemize}

All cases can be dealt with using similar arguments, we spell out the proof of the last case, as that is the most complicated one, see Figure \ref{fig:Lipschitz_proof} for a schematic of the various relevant points and elements of $\mathcal Y$.

\begin{figure}[ht]
    \centering
    \includegraphics[width=\textwidth]{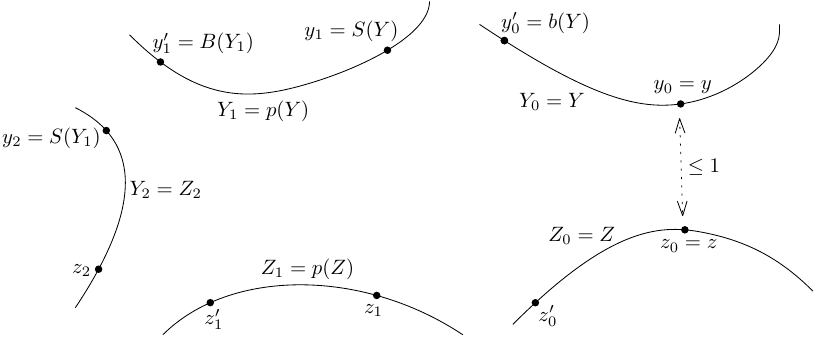}
    \caption{Proof that $f$ is Lipschitz. All pairs of points that lie in the same element of $\mathcal Y$ turn out to be close to each other, due to projections being Lipschitz.}
    \label{fig:Lipschitz_proof}
\end{figure}

Note that $Z_0,Z_1\notin \prel{W_0}{Y}$ by assumption.

Define $y_0=y$, $y'_i=b(Y_i)$ and, for $i=1,2$, $y_i=s(Y_{i-1})$, so that $y_i,y'_i\in Y_i$. Define $z_i,z'_i$ similarly. Note that by construction of $f$ we have $f(b(W))=f(s(W))$ for any $W$, and in particular we have $f(y'_i)=f(y_{i+1})$. Note that $d(y_0,y'_0)\leq 100B$, for otherwise, since $\pi_Y(Z)$ is close to $y_0$ (by the coarse Lipschitz and bounded projections properties), we would have $d_Y(Y_1,Z)\geq 10B$, and so $Y\in \prel{W_0}{Z}$. Similarly, we must have $d(y_1,y'_1)\leq 100B$, for otherwise, since $\pi_{Y_1}(Z)$ is close to $\pi_{Y_1}(Y)$ (again by the same properties), we would have $Y_1\in \prel{Y_0}{Y}$. Similar observations also apply to the $z_i$ and $z'_i$. Finally, $d(y_2,z_2)\leq 100 B$, since $y_2$ coarsely coincides with $\pi_{Y_2}(y)$ (by definition of the order), $y_2$ coarsely coincides with $\pi_{Y_2}(z)$, and $\pi_{Y_2}$ is coarsely Lipschitz.

Putting all these together we get
\begin{align*}
    |f(y)-f(z)|\leq& \left|f(y)-f(y'_0)\right|+|f(y_1)-f(y'_1)|+|f(y_2)-f(z_2)|\\
    &+|f(z_1)-f(z'_1)|+|f(z)-f(z'_0)|,
\end{align*}
and all terms are uniformly bounded, concluding the argument.

\end{proof}

\section{Hyperbolically embedded subgroups}\label{sec:hypembed}

In this section, we characterise hyperbolically embedded subgroups via $\ell^\infty$\nobreakdash-\hspace{0pt}cohomology. We will however need a hypothesis on the ambient group, which is a form of finiteness of homological isoperimetric function for an infinite generating set. It should be seen as a generalisation of finite presentation in two ways, namely it is a homological analogue, and it allows for infinite generating sets. 

Let $G$ be a group generated by a (possibly infinite) generating set $S$.
We denote by $\cay{G,S}$ the Cayley graph of $G$ with respect to $S$. The condition we will need is the following:

\begin{definition}
\label{defn:S-bounded}
    
    $G$ has $S-$bounded $H^2$ if $Z_1(\cay{G,S})$ (meaning cellular cycles with integer coefficients) is generated, as a $\mathbb ZG$-module, by a collection $\mathcal C$ of cycles of bounded $\ell^1$-norm, and moreover there exists a function $\Delta$ such that if $c\in Z_1(\cay{G,S})$ has $\ell^1$-norm at most $n$, then we can write $c=\sum_{i=1}^k g_ic_i$ for some $c_i\in \mathcal C,g_i\in G$ and $k\leq \Delta(n)$.
\end{definition}

\begin{remark}
\label{rem:fin_pres}
If $G$ is finitely presented and $S$ is any finite generating set, then $G$ has $S$-bounded $H^2$. Indeed, we can take the collection of chains from the definition to correspond to relators, and $\Delta(n)$ is the Dehn function.
\end{remark}

Henceforth, we work in the situation described in \Cref{assump}.
\begin{assumption}\label{assump}
    Let $G$ be a group, $S$ be a generating set of $G$ and $H_1,\dots,H_n$ be a finite family of finitely generated subgroups of $G$.
    The generating set $S$ is not assumed to be finite.
    We assume that, for every $i$, the set $S\cap H_i$ generates $H_i$.
\end{assumption}

When we work under \Cref{assump}, we denote by $X$ the Cayley graph of $G$ with respect to $S$, and let $\mathcal{Y}$ be the collection of full subgraphs corresponding to the cosets $gH_i$.
    As in \Cref{sec:relative}, $\mathcal H_\mathcal{Y}$ denotes the collection of horoballs $\mathcal{H}_Y \subseteq \cusp{X,\mathcal{Y}}$.
    We denote by $d_S$ the usual metric on $X$.

\begin{theorem}\label{thm:proper_fif}
  We work under \Cref{assump}.
  Suppose that:
    \begin{itemize}
        \item For every $i \in \{1,\dots,n\}$, the restriction $d_S|_{H_i}$ is proper;
        \item $\cusp{X,\mathcal{Y}}$ has finite homological isoperimetric function.
    \end{itemize}
    Then, the following conditions are equivalent:
    \begin{enumerate}
        \item \label{it:hyp_emb} The family of subgroups $\{H_i\}_{i=1}^n$ is hyperbolically embedded in $(G,S)$;
        \item \label{it:vanishing_V} $\infcohom{X,\mathcal{Y}}{V}=0$ for all 1-injective Banach spaces $V$;
        \item \label{it:vanishing_N} $\infcohom{X,\mathcal{Y}}{\ell^{\infty}(\mathbb N, \R)}=0$.
    \end{enumerate}    
\end{theorem}
\begin{proof}
    By \cite[Theorem 1.1, Theorem 1.2]{metric_rel_hyp}, $\{H_i\}$ is hyperbolically embedded in $(G,S)$ if and only if $\cusp{X,\mathcal Y}$ is hyperbolic and $d_S|_{H_i}$ is proper for every $i$. (When the $H_i$ are hyperbolic, this can also be deduced from \cite[Theorem 3.9]{AMS}, saying that the $H_i$ are quasiconvex and geometrically separated in $X$, together with the arguments showing that a hyperbolic group is hyperbolic relative to any almost malnormal finite collection of quasiconvex subgroups \cite{Bow:rel_hyp_PUB}. Note that this can also be proven with the more modern approach of ``guessing geodesics'' \cite[Theorem 3.11]{MS:disk}.)

    In order to compute $H^2_\infsub(X,\mathcal{Y};V)$, we need a pair quasi-isometric to $(X,\mathcal{Y})$ in which the subgraphs are disjoint.
    We consider the ``truncation'' of the cusped space $\cusp{X,\mathcal{Y}}$ in which we attach to $X$ only the first layer of the horoballs; these first layers are identified naturally with the elements of $\mathcal{Y}$.
    We denote the resulting graph and the corresponding family of subgraphs by $(\hat X, \hat{\mathcal{Y}})$, which is a pair quasi-isometric to $(X,\mathcal Y)$.

    $\eqref{it:hyp_emb}\implies\eqref{it:vanishing_V}$.
    Suppose that $\{H_i\}$ is hyperbolically embedded in $G$.
    By \Cref{prop:cusp_rel_cohomo_0}, we have $\infcohom{\cusp{X,\mathcal Y},\mathcal H_\mathcal Y}{V}=0$ for all 1-injective Banach spaces $V$.
    Using that for every $i$ the restricted metric $d_S|_{H_i}$ is proper, we see that the natural inclusion $(\hat X,\hat{\mathcal Y})\to (\cusp{X,\mathcal Y},\mathcal H_{\mathcal Y})$ satisfies the assumptions of \Cref{prop:excision}, and therefore we conclude that $\infcohom{X,\mathcal{Y}}{V}=\infcohom{\hat X,\hat{\mathcal{Y}}}{V} = 0$, as required.

    $\eqref{it:vanishing_N}\implies\eqref{it:hyp_emb}$.
    Suppose now that $\infcohom{X,\mathcal{Y}}{\ell^{\infty}(\mathbb N, \R)}=0$.
    Again using \Cref{prop:excision}, we also get $\infcohom{\cusp{X,\mathcal Y},\mathcal H_\mathcal Y}{V}=0$ for all 1-injective Banach spaces $V$.
    Then, by \Cref{prop:rel_cohom_surj}, vanishing also holds for $\infcohom{\cusp{X,\mathcal Y}}{V}$ since horoballs are uniformly hyperbolic.
    We conclude by applying Proposition \ref{prop:vanishing_hyp}.
\end{proof}

\begin{thm}\label{thm:Sbounded}\label{thm:hyp_emb}
  We work under \Cref{assump}.
  Suppose that:
    \begin{itemize}
        \item For every $i \in \{1,\dots,n\}$, the restriction $d_S|_{H_i}$ is proper;
        \item $G$ has $S$-bounded $H^2$.
    \end{itemize}
    Then, conditions \eqref{it:hyp_emb}, \eqref{it:vanishing_V} and \eqref{it:vanishing_N} of \Cref{thm:proper_fif} are equivalent.
\end{thm}

\begin{proof}
  We show that $\cusp{X,\mathcal{Y}}$ has finite homological isoperimetric function, so that we can conclude by applying \Cref{thm:proper_fif}.

      Consider a closed path $p$ in $\cusp{X,\mathcal Y}$. If the path is entirely contained in a horoball, then it can be filled in a controlled way. That is, for all sufficiently large $R$, there exists a function $\theta$ such that $\mbox{Area}_R(p) \le \theta(\length{p})$, where $R$ and $\theta$ do not depend on $p$.
      Suppose now that $p$ is contained in $X$. For $\mathcal C$ the set as in the definition of $S-$bounded $H^2$, there exists $R_0\geq 0$ such that for all $R\geq R_0$ we have $\norm{c}_F^R \le \norm{c}_1$ for all $c\in\mathcal C$; this is because we can fill $c$ by ``coning over'' one of its vertices.
      Set $M=\max_{c\in\mathcal C} \{\lvert c\rvert_1\}$. Back to our closed path $p$, we can write $\chain{p}=\sum_{i=1}^k g_ic_i$ for some $c_i\in \mathcal C$ and $k\leq \Delta(\lvert c\rvert_1)$, and hence we have $\norm{\chain{p}}_F^R \le \Delta(\length{p}) M$, so that also in this case $\chain{p}$ can be filled in a controlled way.

    Finally, suppose that $p$ is neither contained in a horoball nor in $X$. In this case we can split $p$ (by ``cutting'' it where it passes from $X$ to a horoball or vice versa) into subpaths, each contained either in $X$ or in an ``extended horoball'' (the subgraph spanned by the union of a horoball with the corresponding coset on which the horoball is attached).
    Considering geodesics in the relevant cosets of the $H_i$, we can then write $\chain{p}=\sum_i^k \chain{p_i}$, where each $p_i$ is a closed path which is either contained in an extended horoball or in $X$; we do so by considering the subpaths contained in extended horoballs and closing them up with a geodesic in the corresponding coset (so that the reverses of these geodesics also close up the broken path in $X$). Since $d_S|_{H_i}$ is proper, the geodesics in the cosets have length controlled in terms of the distance (computed in $X$) between their endpoints.
    That is, there is a function $h$, independent of $p$, such that each of these geodesics has length at most $h(D)$, if $D$ is the distance in $X$ between its endpoints.
    We can assume that $h$ is nondecreasing, so that, in particular, all these geodesics have length at most $h(\length{p})$.
    Also, in order to close up the subpaths of $p$, we are adding at most $\length{p}$ of those geodesics;
    therefore, we have $\sum \length{p_i}\leq \length{p}\cdot h(\length{p})$.
    Since we can fill each $\chain{p_i}$ in a controlled way, we can also fill $\chain{p}$, as required.
\end{proof}

\subsection{Relative hyperbolicity}

We now make it explicit why Theorem \ref{thm:hyp_emb} generalises the $\ell^\infty$-characterisation of finitely presented relatively hyperbolic groups from \cite{Mil2021}.

\begin{corollary}[cf.\ \cite{Mil2021}]
\label{cor:Francesco}
    Let $G$ be a finitely presented group. Then $G$ is hyperbolic relative to $\mathcal H$ if and only if $\infcohom{G,\{H_i\}}{V}=0$
    for all injective Banach spaces $V$.
\end{corollary}

\begin{proof}
    Let $S$ be a finite generating set for $G$. By Proposition \ref{prop:group_case} we can relate the cohomology appearing in the statement of \Cref{thm:hyp_emb} to the group-theoretic $\ell^\infty$-cohomology, that is, we have $\infcohom{\cay{G,S},G/\{H_i\}}{V}=\infcohom{G,\{H_i\}}{V}$.
    Moreover, by \cite[Proposition 4.28]{DGO17}, $G$ is hyperbolic relative to $\mathcal H$ if and only if $\mathcal H$ is hyperbolically embedded in $(G,S)$. Since $d_S|_{H_i}$ is proper for all $i$ because $S$ is finite, and $G$ has $S$-bounded $H^2$ by Remark \ref{rem:fin_pres}, we can apply Theorem \ref{thm:hyp_emb}.
\end{proof}

\begin{remark}
    We believe that there should be an improvement of Theorem \ref{thm:hyp_emb} that allows one to fully recover the results of \cite{Mil2021}, which hold more generally for relatively finitely presented groups rather than finitely presented groups. The price to pay is increasing the level of technicality of Definition \ref{defn:S-bounded}, and we do not pursue this here.
\end{remark}

\bibliographystyle{fram_alpha}
\bibliography{bibliography.bib}

\end{document}